\documentclass[11pt]{amsart}

\numberwithin{equation}{section}

\newtheorem{theorem}{Theorem}[section]
\newtheorem{proposition}[theorem]{Proposition}

\newtheorem{lemma}[theorem]{Lemma}
\newtheorem{corollary}[theorem]{Corollary}

\theoremstyle{definition}
\newtheorem{definition}[theorem]{Definition}
\newtheorem{remark}[theorem]{Remark}

\newtheorem{problem}[theorem]{Problem}
\newtheorem{example}[theorem]{Example}

\usepackage{amsfonts,amssymb,amscd,amstext,mathrsfs,color}

\usepackage{graphicx}
\usepackage[dvips]{epsfig}
\usepackage[dvipsnames]{xcolor}
\usepackage[utf8]{inputenc}
\usepackage{hyperref}
\usepackage{verbatim}
\usepackage{enumerate}
\input xy
\xyoption{all}

\newcommand\Bscr{\mathscr{B}}
\newcommand\Cscr{\mathscr{C}}

\newcommand\Fscr{\mathscr{F}}

\newcommand\Lscr{\mathscr{L}}

\newcommand\Oscr{\mathscr{O}}

\newcommand\C{\mathbb{C}}

\newcommand\CP{\mathbb{CP}}

\renewcommand\H{\mathbb{H}}

\newcommand\N{\mathbb{N}}
\renewcommand\P{\mathbb{P}}
\newcommand\R{\mathbb{R}}

\renewcommand\S{\mathbb{S}}

\newcommand\Z{\mathbb{Z}}

\newcommand\igot{\mathfrak{i}}
\newcommand\jgot{\mathfrak{j}}
\newcommand\kgot{\mathfrak{k}}

\renewcommand\imath{\igot}

\newcommand\hra{\hookrightarrow}
\newcommand\lra{\longrightarrow}

\newcommand\wt{\widetilde}

\newcommand\di{\partial}
\newcommand\dibar{{\overline\partial}}

\newcommand\dist{\mathrm{dist}}

\newcommand\Res{\mathrm{Res}}

\newcommand\Aut{\mathrm{Aut}}

\def\dist{\mathrm{dist}}

\def\ker{\mathrm{ker}}

\begin{document}

\title{Holomorphic Legendrian curves in $\CP^3$ \\ 
and superminimal surfaces in $\S^4$}

\author{Antonio Alarc\'on, Franc Forstneri\v c, and Finnur L\'arusson}


\date{28 October 2019.  Minor changes 7 September 2020 and 7 February 2022}


\begin{abstract}   
We obtain a Runge approximation theorem for holomorphic Legendrian curves
and immersions in the complex projective $3$-space $\CP^3$, both from open and compact Riemann surfaces,  
and we prove that the space of Legendrian immersions from an open Riemann surface into $\CP^3$ is path
connected. We also show that holomorphic Legendrian immersions from Riemann surfaces of finite genus 
and at most countably many ends, none of which are point ends, satisfy the Calabi--Yau property. 
Coupled with the Runge approximation theorem, we infer that every open Riemann surface
embeds into $\CP^3$ as a complete holomorphic Legendrian curve.
Under the twistor projection $\pi:\CP^3\to \S^4$ onto the $4$-sphere, immersed holomorphic 
Legendrian curves $M\to \CP^3$ are in bijective correspondence with superminimal 
immersions $M\to\S^4$ of positive spin according to a result of Bryant.  
This gives as corollaries the corresponding results on superminimal surfaces in $\S^4$.
In particular, superminimal immersions into $\S^4$ satisfy the Runge approximation theorem 
and the Calabi--Yau property.
\end{abstract}

\maketitle

%
%
\section{Introduction} 
It is well known that the 3-dimensional complex projective space $\CP^3$ admits a unique complex contact
structure, that is to say, a completely noninvolutive holomorphic hyperplane subbundle $\xi$ 
of the tangent bundle $T\CP^3$ such that any other holomorphic contact bundle on $\CP^3$ is 
contactomorphic to $\xi$ by an automorphism of $\CP^3$ 
(see C.\ LeBrun and S.\ Salamon \cite{LeBrunSalamon1994IM,LeBrun1995IJM}). This contact structure is
determined by the following homogeneous $1$-form on $\C^4$ 
via the standard projection $\C^4\setminus \{0\}\to \CP^3$:
\begin{equation}\label{eq:alpha0}
	\alpha_0=z_0dz_1-z_1dz_0+z_2dz_3-z_3dz_2.
\end{equation}
(See Sect.\ \ref{sec:F}.)  Uniqueness makes this contact structure 
fundamentally interesting. This was amplified in 1982 when R.\ Bryant \cite{Bryant1982JDG}  
discovered that the Penrose {\em twistor projection} $\pi:\CP^3\to \S^4$  
(a fibre bundle projection onto the $4$-sphere whose fibres are projective lines) 
induces a  bijective correspondence between immersed holomorphic Legendrian curves in $\CP^3$ 
and immersed superminimal surfaces of positive spin in $\S^4$. (When speaking of the $4$-sphere,
we always consider it endowed with the spherical metric induced by the Euclidean metric
on the unit sphere $\S^4\subset\R^5$.)
Furthermore, the contact bundle $\xi$ on $\CP^3$ is the orthogonal complement 
of the vertical tangent bundle of $\pi$ in the Fubini-Study metric, and the differential 
$d\pi$ maps $\xi$ isometrically onto $T\S^4$, so $\pi$ maps Legendrian curves
locally isometrically to superminimal surfaces in $\S^4$. The latter form an interesting subclass of 
the class of all minimal surfaces in $\S^4$. 
Bryant proved in \cite[Theorem F]{Bryant1982JDG} that for any pair of meromorphic functions $f,g$ 
on a Riemann surface $M$ with $g$ nonconstant, the map given in homogeneous coordinates by 
\begin{equation}\label{eq:B1}
	\Bscr(f,g) = \big[dg: fdg-\tfrac 1 2 g df: g dg: \tfrac 1 2 df\big]:M\lra \CP^3 
\end{equation}
is a holomorphic Legendrian curve in $\CP^3$.  
Using this formula, he showed that any compact Riemann surface $M$ admits a holomorphic 
embedding into $\CP^3$ as a Legendrian curve (see \cite[Theorem G]{Bryant1982JDG}), 
and he inferred that any such $M$ admits a conformal, generically injective 
immersion $M\to\S^4$ onto a superminimal surface in $\S^4$ (see \cite[Corollary H]{Bryant1982JDG}).

In the present paper we go considerably further by treating not only Legendrian curves in $\CP^3$
and conformal superminimal surfaces in $\S^4$ parameterised by compact Riemann surfaces, 
but also those parameterised by open or by compact bordered Riemann surfaces. 
In particular, we obtain 
the first general existence and approximation results in the 
literature for complete noncompact superminimal surfaces in the $4$-sphere. More about this below.

Let us now describe the contents of the paper. 

We begin by presenting in Sect.\ \ref{sec:F} a unified approach from first principles
to a couple of representation formulas for Legendrian curves in $\CP^3$, 
the one of Bryant \eqref{eq:B1} and another one adapted from the recent papers 
by Alarc\'on, Forstneri\v c and L\'opez \cite{AlarconForstnericLopez2017CM} 
and Forstneri\v c and L\'arusson \cite{ForstnericLarusson2018MZ}; see \eqref{eq:Bfg}, \eqref{eq:Fhg}. 
The relationship between them is given by \eqref{eq:comparison}. As pointed out in Remark \ref{rem:comparison}, 
the optimal choice of a formula to use depends on the particular problem one wants to solve. 
Although each formula has a set of exceptional curves it does not cover, any given Legendrian curve is 
nonexceptional in some homogeneous coordinates on $\CP^3$  (see Proposition \ref{prop:all}). 
By choosing homogeneous coordinates on $\CP^3$ so that the hyperplane at infinity 
intersects our Legendrian curve transversely, which is possible by Bertini's theorem, 
the two meromorphic functions determining the curve have only simple poles.  
This condition means that the curve is immersed near the poles.

In Sect.\ \ref{sec:approx-interpol} we use Bryant's formula \eqref{eq:B1} to
prove the Runge approximation theorem coupled with the Weierstrass interpolation theorem 
for holomorphic Legendrian curves in $\CP^3$, both from compact and open Riemann surfaces 
(see Theorem \ref{th:Runge}), as well as the corresponding result for holomorphic Legendrian immersions 
(see Theorem \ref{th:Runge-for-immersions}). For open Riemann surfaces, we also
have a Runge approximation theorem for Legendrian embeddings into $\CP^3$
(see Corollary \ref{cor:Runge-embeddings}), where by an embedding we mean an injective immersion.

In Sect.\ \ref{sec:connected} we use the second representation formula \eqref{eq:Fhg} 
to prove that the space of all Legendrian immersions $M\to\CP^3$ from an arbitrary 
open Riemann surface is path connected; see Theorem \ref{th:connected}.
On the other hand, the space of Legendrian immersions is not connected if $M$ is compact.  
It will split into components by degree, and perhaps further.

These results imply that every formal Legendrian immersion from an open Riemann surface 
to $\CP^3$ can be deformed to a genuine holomorphic Legendrian immersion, unique up to homotopy
(see Theorem \ref{th:h-principle}). It remains an open problem whether the inclusion of the space
of holomorphic Legendrian immersions $M\to \CP^3$ into the space of formal Legendrian immersions
satisfies the full parametric h-principle. For immersed Legendrian curves in $\C^{2n+1}$ 
with its standard contact structure the parametric h-principle was proved in \cite{ForstnericLarusson2018MZ};
however, the technical problems that arise for Legendrian curves in projective spaces are considerable.

In Sect.\ \ref{sec:CY} we introduce an axiomatic approach to the {\em Calabi--Yau problem}
which unifies recent results in this direction in various geometries. 
The motivation behind results of this type is the {\em Calabi--Yau problem for minimal hypersurfaces},
asking whether there exist complete bounded minimal hypersurfaces in $\R^n$ for $n\ge 3$. This problem
originates in Calabi's conjecture from 1965 that such hypersurfaces do not exist 
(see \cite[p.\ 170]{Calabi1965Conjecture}). Nothing seems known about this question
concerning hypersurfaces in $\R^n$ for $n\ge 4$. However, several constructions of complete bounded 
$2$-dimensional minimal surfaces in $\R^n$ for any $n\ge 3$ have been developed, 
starting with the seminal works of L.\ Jorge and F.\ Xavier \cite{JorgeXavier1980AM} in 1980 
and N.\ Nadirashvili \cite{Nadirashvili1996IM} in 1996. (Note that we are not talking of hypersurfaces,
unless $n=3$.) 
Subsequent developments were inspired by S.-T.\ Yau's {\em 2000 Millennium Lecture} \cite{Yau2000AMS} 
where he revisited Calabi's conjectures and proposed several questions concerning topology, complex structure, 
and boundary behaviour of complete bounded minimal surfaces in $\R^3$. A recent survey of this topic
can be found in Alarc\'on and Forstneri\v c \cite[Sect.\ 5.3]{AlarconForstneric2019JAMS}; see also the paper 
\cite{AlarconForstneric2019RMI} where the Calabi--Yau theorem was established for immersed minimal
surfaces in $\R^n$, $n\ge 3$, from any open Riemann surface of finite genus and at most countably many
ends, none of which are point ends. 

Recently, the Calabi--Yau phenomenon has been discovered in other geometries, 
and it is reasonable to expect that more examples will follow. This motivated us to formulate an axiomatic 
approach by introducing the {\em Calabi--Yau property} which a class of immersions into a given 
Riemannian manifold $N$ (or a class of manifolds) may or may not have; see Definition \ref{def:CY} 
and Theorem \ref{th:abstractCY}. This property means that one can enlarge the intrinsic diameter of 
an immersed manifold as much as desired by $\Cscr^0$ small perturbations of the immersion
in the given class. Combining the Calabi--Yau property with Runge's approximation property for immersions 
of the given class into the manifold $N$ (see Definition \ref{def:Runge}) gives complete immersions of 
this class from all open admissible manifolds into $N$ (see Theorem \ref{th:complete2}). 

As a particular case of interest, we discuss Legendrian immersions. It was proved in 
\cite{AlarconForstneric2019IMRN} that holomorphic Legendrian immersions from bordered Riemann surfaces into
any complex contact manifold with an arbitrary Riemannian metric enjoy the Calabi--Yau property. We show that 
one can at the same time interpolate the given map at finitely many points 
(see Corollary \ref{cor:CYLegendrian}). Coupled with the Runge approximation 
theorem for Legendrian embeddings of open Riemann surfaces into $\CP^3$ 
(see Corollary \ref{cor:Runge-embeddings}) and Bryant's Legendrian embedding theorem for compact
Riemann surfaces \cite[Theorem G]{Bryant1982JDG}, it follows that every Riemann surface embeds 
into $\CP^3$ as a complete holomorphic Legendrian curve (see Corollary \ref{cor:completeLeg}).

In Sect.\ \ref{sec:S4} we apply our results to the study of superminimal surfaces in the $4$-sphere, $\S^4$,
endowed with the spherical metric. It follows in particular that the Runge approximation theorem and the 
Weierstrass interpolation theorem hold for conformal superminimal immersions of Riemann surfaces 
(both open and closed) into $\S^4$,  and every open Riemann surface is the conformal structure of 
a complete conformally immersed superminimal surface in $\S^4$ (see Corollaries \ref{cor:RungeS4} and 
\ref{cor:WeierstrassS4}). Furthermore, 
any smooth conformal superminimal immersion $M \to \S^4$ from a compact bordered Riemann surface 
can be approximated as closely as desired uniformly on $M$  by a continuous map 
$M\to \S^4$ whose restriction to the interior of $M$ is a complete conformal superminimal surface 
with Jordan boundary (see Theorem \ref{th:CYS4}). The analogous result for minimal surfaces
in $\R^n$, $n\ge 3$, with the Euclidean metric was proved in 
\cite{AlarconDrinovecForstnericLopez2015PLMS}. Finally, for every open Riemann surface, $M$, 
the spaces of conformal superminimal immersions $M\to\S^4$ of positive or negative spin are path connected; 
see Corollary \ref{cor:connectedS4}.

Results of this paper concerning holomorphic Legendrian curves in $\CP^3$ 
can be generalised to higher dimensional projective spaces $\CP^{2n+1}$ with the unique
holomorphic contact structure determined by the following homogeneous 
holomorphic $1$-form on $\C^{2n+2}$:
\[
	\alpha_0 = \sum_{j=0}^n z_{2j}dz_{2j+1} -  z_{2j+1}dz_{2j}.
\]
Since $\CP^{2n+1}$ is the twistor space of the quaternionic projective space $\H\P^n$ 
(see \cite[p.\ 113]{LeBrunSalamon1994IM}), this gives similar applications to superminimal
surfaces in $\H\P^n$ for $n>1$. 
We shall not give the details of this generalisation because this would considerably enlarge the paper
without providing any substantially new ideas or techniques. 
After the completion of this paper, the approach developed here was used by the second named author
in \cite{Forstneric2020JGEA} to establish the Calabi--Yau property of superminimal surfaces of appropriate spin 
in any self-dual or anti-self-dual Einstein four-manifold, the four-sphere being a special case.

%
%
%
%
\section{Representation formulas for Legendrian curves in $\CP^3$}\label{sec:F}
Let $\alpha_0$ be the homogeneous $1$-form on $\C^4$ defined by \eqref{eq:alpha0}.
Its differential is the standard complex symplectic form on $\C^4$.
At each point $z=(z_0,z_1,z_2,z_3)\in \C^4\setminus \{0\}$, $\ker \, \alpha_0(z)$ 
is a complex hyperplane in $T_z\C^4$ containing the radial vector $\sum_{i=0}^3 z_i \frac{\di}{\di z_i}$. 
Let $\pi:\C^4\setminus \{0\}\to \CP^3$ be the standard projection and $[z_0:z_1:z_2:z_3]$ be the
homogeneous coordinates on $\CP^3$. Since $\alpha_0$ is homogeneous, 
there is a unique holomorphic hyperplane subbundle $\xi\subset T\CP^3$ defined by the condition
\[
	\left\{v\in T_z\C^4: d\pi_z(v)\in \xi_{\pi(z)} \right\}= \ker\,\alpha_0(z),\quad z\in \C^4\setminus \{0\}. 
\]
It turns out that $\xi$ is a holomorphic contact bundle on $T\CP^3$, and the essentially unique one (see \cite{LeBrunSalamon1994IM} or \cite[Proposition 2.3]{LeBrun1995IJM}). 
The following lemma shows that the restriction of $\xi$ to any affine chart $\C^3\subset \CP^3$ 
is linearly contactomorphic to the standard contact structure on $\C^3$.

%
%
\begin{lemma}\label{lem:standardform}
For every projective hyperplane $\CP^2\cong H\subset \CP^3$ there are linear coordinates
$(z'_1,z'_2,z'_3)$ on $\C^3=\CP^3\setminus H$ in which $\xi$ is defined by the contact form
\begin{equation}\label{eq:alpha}
	\alpha=dz'_1 + z'_2 dz'_3 - z'_3dz'_2.
\end{equation}
\end{lemma}

Note that every linear automorphism of $\C^3$ extends to a unique projective automorphism
of $\CP^3$. Hence, in the context of the lemma there exists $\phi\in \Aut(\CP^3)$ such that
$\phi(H)=H$ and $\phi_*(\xi)=\ker\,\alpha$ on $\C^3=\CP^3\setminus H$.

\begin{proof}
Due to symmetries of $\alpha_0$ as defined in \eqref{eq:alpha0}, it suffices to consider hyperplanes $H\subset \CP^3$ 
of the form $z_0 = a_1z_1+a_2z_2 +  a_3z_3$ for some $a_1,a_2,a_3\in\C$. 
The affine chart $\CP^3\setminus H=\C^3$ is then determined by the affine hyperplane 
\[
	\Lambda = \{z_0 = 1+a_1z_1+a_2z_2 +  a_3z_3\}\subset \C^4.
\]
Note that $(z_1,z_2,z_3)$ are affine coordinates on $\Lambda$, and the restriction
of $\alpha_0$ to it is 
\[
	\alpha=(1+a_2z_2 +  a_3z_3)dz_1-(z_3+a_2z_1)dz_2+(z_2-a_3z_1)dz_3.
\]
We introduce new linear coordinates on $\C^3$ by 
\[
	z'_1=z_1,\quad z'_2=z_2-a_3z_1,\quad z'_3=z_3 + a_2z_1.
\]
Then, 
\begin{eqnarray*}
	(1+a_2z_2 +  a_3z_3)dz_1 &=& (1+a_2z'_2+a_3z'_3)dz_1, \\
	(z_3+a_2z_1)dz_2 &=&  z'_3 (dz'_2+a_3 dz_1) = z'_3 dz'_2+a_3z'_3 dz_1, \\
	(z_2-a_3z_1)dz_3  &=& z'_2(dz'_3-a_2dz_1) = z'_2dz'_3 - a_2z'_2dz_1,
\end{eqnarray*}
and hence $\alpha$ is given in these coordinates by \eqref{eq:alpha}.
\end{proof}

Lemma \ref{lem:standardform} shows that for any hyperplane $H\subset \CP^3$, homogeneous 
coordinates on $\CP^3$ can be chosen such that $H=\{z_0=0\}$ and the contact structure $\xi$
is given on $\C^3= \CP^3\setminus H$ as the kernel of the holomorphic contact form
\begin{equation}\label{eq:alphast}
	\alpha=dz_1 + z_2dz_3-z_3dz_2,\quad\  \alpha\wedge d\alpha = 2dz_1\wedge dz_2\wedge dz_3\ne 0.
\end{equation}
Globally on $\CP^3$, $\alpha$ is a meromorphic $1$-form with a second order pole 
along the hyperplane $H=\{z_0=0\}$. It can be viewed as a nowhere vanishing 
holomorphic contact $1$-form on $\CP^3$ with values in the normal line bundle $L=T\CP^3/\xi$ 
of the contact structure. (See \cite[Sect.\ 2]{LeBrunSalamon1994IM} for the precise explanation.) 
Furthermore, $\omega=\alpha\wedge d\alpha$ is a holomorphic $3$-form on $\CP^3$ 
with values in the line bundle $L^2$, hence an element of $H^0(\CP^3,K \otimes L^2)$
where $K=\Lambda^3(T^*\CP^3)$ is the canonical bundle of $\CP^3$.  
Being nowhere vanishing, $\omega$ defines a holomorphic trivialisation of 
$K\otimes L^2$, so we infer that $L\cong K^{-1/2}=\Oscr_{\CP^3}(2)$. 
In other words, the dual bundle $L^*=L^{-1}$ is the square of the universal bundle on $\CP^3$. 

We also consider the contact form on $\C^3$ given by
\begin{equation}\label{eq:betast}
	\beta=dz_1+z_2\, dz_3,
\end{equation}
with $\beta\wedge d\beta = dz_1\wedge dz_2\wedge dz_3$. The map $\psi:\C^3\to\C^3$ defined by
\begin{equation}\label{eq:psi}
	\psi(z_1,z_2,z_3) = \left(z_1+\frac{z_2z_3}{2}, z_3, -\frac{z_2}{2}\right)
\end{equation}
is a polynomial automorphism of $\C^3$, and a simple calculation shows that $\psi^*\alpha =\beta$.
It follows that $\psi$ maps $\beta$-Legendrian curves to $\alpha$-Legendrian curves.
Clearly, we can represent $\beta$-Legendrian curves in either of the following two forms:
\begin{equation}\label{eq:case1}
	z_1=f, \quad z_2=-\frac{df}{dg},\quad z_3=g,
\end{equation}
\begin{equation}\label{eq:case2}
	z_1=-\int hdg,\quad z_2=h,\quad z_3=g,
\end{equation}
where $f,g,h$ are meromorphic functions on a given Riemann surface $M$. 
(The part of the curve contained in $\C^3$ is the image of the complement $M\setminus P$
of the set $P$ of poles of the respective pair of functions $(f,g)$ or $(h,g)$.)

In the first case \eqref{eq:case1}, the pair of functions $(f,g)$ is arbitrary subject only to the condition 
that $g$ is nonconstant. The exceptional family of Legendrian lines with $z_1=const.$, $z_3=const.$
cannot be represented in this way. 

In the second case \eqref{eq:case2}, the pair $(h,g)$ must be such that $hdg$ is an exact 
meromorphic $1$-form, which therefore has a meromorphic primitive $f=-\int hdg$ determined up 
to an additive constant. We discuss this condition in Proposition \ref{prop:exact}.
Conversely, assuming that $g$ is nonconstant, we can express $h$ in terms of $f$ by $h=-df/dg$. 

Applying the automorphism $\psi\in\Aut(\C^3)$, given by \eqref{eq:psi}, to $\beta$-Legendrian curves
\eqref{eq:case1}, \eqref{eq:case2} yields the following formulas for $\alpha$-Legendrian curves in $\CP^3$:
\begin{eqnarray}\label{eq:Bfg}
	\Bscr(f,g) &=& \left[1:f-\frac{1}{2}g\frac{df}{dg}: g: \frac{1}{2} \frac{df}{dg} \right]
	= \big[dg: fdg-\tfrac 1 2 g df: g dg: \tfrac 1 2 df\big], \\
	\label{eq:Fhg}
	\Fscr(h,g) &=& \left[1: \frac{hg}{2} -\int h dg: g: -\frac{h}{2}\right]
	=\left[1: \int g dh - \frac{hg}{2} : g: -\frac{h}{2}\right].
\end{eqnarray}
Both formulas depend on the choice of homogeneous coordinates and are related by
\begin{equation}\label{eq:comparison}
	\Bscr(f,g)=\Fscr(h,g),\ \ \text{where}\ \ f=-\int hdg\ \ \text{and}\ \ h=-df/dg.
\end{equation}
The formula \eqref{eq:Bfg} was used by Bryant \cite{Bryant1982JDG} to prove that every
compact Riemann surface embeds in $\CP^3$ as a holomorphic Legendrian curve. 
The second formula \eqref{eq:Fhg} has been exploited in the study of Legendrian curves in $\C^3$ in  
the recent work \cite{AlarconForstnericLopez2017CM}. 

The family of exceptional $\beta$-Legendrian lines $z_1=a=const.$, $z_2=2t\in\C$,
$z_3=b=const.$ is mapped by the automorphism $\psi$ given in \eqref{eq:psi} to the family of 
exceptional $\alpha$-Legendrian lines 
\begin{equation}\label{eq:exceptional}
	\left[1: a+bt: b: -t \right] \quad \text{with}\ t\in\CP^1\ \text{and}\ a,b\in\C,
\end{equation}
which are not of the form $\Bscr(f,g)$.
On the other hand, every Legendrian curve intersecting this affine chart equals
$\Fscr(h,g)$ for a unique pair of meromorphic functions $(h,g)$ and a choice of an additive 
constant determining the value of the integral $\int hdg$ at an initial point $p_0\in M$.

We now show that every nonconstant Legendrian curve in $\CP^3$ is of the form 
$\Bscr(f,g)$ and $\Fscr(h,g)$ in some homogeneous coordinate system on $\CP^3$.

%
%
\begin{proposition}\label{prop:all}
Let $F:M\to \CP^3$ be a nonconstant holomorphic Legendrian curve
from an open or compact Riemann surface $M$.
\begin{enumerate}[\rm (a)]
\item
There are homogeneous coordinates on $\CP^3$ such that $F=\Bscr(f,g)$
(see \eqref{eq:Bfg}), where $f$ and $g$ are meromorphic functions on $M$ with only simple poles.
\item
There are homogeneous coordinates on $\CP^3$ such that $F=\Fscr(h,g)$ (see \eqref{eq:Fhg}),
where $h$ and $g$ are meromorphic functions on $M$ with only simple poles.
\end{enumerate}
Furthermore, every Legendrian curve $M\to\CP^3$ given by \eqref{eq:Bfg} or \eqref{eq:Fhg},
with the functions $f,g,h$ having only simple poles, is an immersion on a neighbourhood 
of the union of the sets of poles of $f$ and $g$ (for \eqref{eq:Bfg}), or $h$ and $g$ (for \eqref{eq:Fhg}). 
\end{proposition}

\begin{proof} 
Let $F:M\to\CP^3$ be a nonconstant holomorphic Legendrian curve. In view of 
E.\ Bertini's theorem (see e.g.\ \cite[p.\ 150]{GoreskyMacPherson1988} or \cite{Kleiman1974}
and note that this is essentially an application of the transversality theorem), 
$F$ intersects most complex hyperplanes $H\subset\CP^3$ transversely. 
Fix such $H$ and choose homogeneous coordinates $[z_0:z_1:z_2:z_3]$ on $\CP^3$ with $H=\{z_0=0\}$ 
and so that the contact form on $\CP^3\setminus H=\C^3$ is given by \eqref{eq:alphast}.
The preimage $F^{-1}(H)=\{p\in M:F(p)\in H\}$ is then a closed discrete subset of $M$. 
Hence, we can represent $F$ in either form \eqref{eq:Bfg} or \eqref{eq:Fhg}, the only exceptions 
being the family of projective lines \eqref{eq:exceptional} which cannot be represented by 
Bryant's formula \eqref{eq:Bfg}. We shall deal with this issue later.

Consider a point $p\in F^{-1}(H)$. Choose a local holomorphic coordinate $\zeta$ on $M$ 
with $\zeta(p)=0$. Write $F=[1:F_1:F_2:F_3]$ and let $k\in \N$ be the maximal order of poles at $p$ 
of the components $F_1$, $F_2$ and $F_3$. Multiplying by $\zeta^k$ we obtain 
\[
	F(\zeta)=\bigl[\zeta^k: \zeta^k F_1(\zeta):\zeta^k F_2(\zeta):\zeta^k F_3(\zeta)\bigr],
\]
where the functions $\zeta^k F_j(\zeta)$ for $j\in \{1,2,3\}$ are regular at $\zeta=0$ and 
at least one of them is nonvanishing at $\zeta=0$. Looking at the map $F$ in the corresponding 
affine chart $\{z_j=1\}$, we see that $F$ is transverse to $H$ at the point $p$ if and only if 
the derivative $dz_0/d\zeta$ is nonvanishing at $\zeta=0$, which holds if and only if $k=1$. 
Inspection of the formulas for $\Bscr(f,g)$ and $\Fscr(h,g)$ then shows that the functions 
$f,g$ or $h,g$ have at most simple poles at $p$.
Conversely, the above argument shows that the intersection of $F$ with $H$ is transverse 
at any simple pole of the functions $f$ and $g$, or $h$ and $g$. In particular, $F$ is an immersion 
near such points.

%
%
It remains to show that the exceptional lines \eqref{eq:exceptional} become 
nonexceptional in another coordinate system. Consider the following coordinates on $\CP^3$: 
\[
	z'_0=z_0,\ \ \ z'_1=z_1,\ \ \ z'_2=-z_3,\ \ \ z'_3=z_2.
\]
We have not changed $H=\{z_0=0\}$, so we are still in the same affine chart.
In these coordinates, the form $\alpha_0$ \eqref{eq:alpha0} restricted to the affine chart $\{z_0=1\}$ equals
\[
	\alpha= dz'_1 + z'_2 dz'_3 - z'_3 dz'_2,
\]
and the exceptional family of lines is given in the new coordinates by 
\[
	\left[1: a+bt: t:b  \right] \quad \text{for}\ t\in\CP^1\ \text{and}\ a,b\in\C.
\]
This curve equals $\Bscr(f,g)$ with $f(t)=a+2bt$ and $g(t)=t$.
This shows that every nonconstant Legendrian curve in $\CP^3$ is of the form $\Bscr(f,g)$ in 
some homogeneous coordinate system. 
\end{proof}

There are Legendrian immersions \eqref{eq:Bfg}, \eqref{eq:Fhg} 
given by functions $f,h,g$ with higher order poles. However, this means that the hyperplane $H$ 
determining the affine chart was not well chosen, and a small deformation of it yields a
representation by functions with simple poles. 

The following is an immediate corollary to Proposition \ref{prop:all}.

%
%
\begin{corollary}\label{cor:immersions}
Let $F=\Fscr(h,g):M\to\CP^3$ be a holomorphic Legendrian curve of the form
\eqref{eq:Fhg} with $g,h$ having only simple poles. Then, $F$ is an immersion if and only if 
$(h,g):M\setminus P\to\C^2$ is an immersion, where $P=P(h)\cup P(g)$ is the union of polar loci of $h$ and $g$.
\end{corollary}

\begin{proof}
By Proposition \ref{prop:all}, $F$ is an immersion if and only if its restriction $M\setminus P\to\C^3$ 
is an immersion. This restriction is equivalent to the $\beta$-Legendrian curve \eqref{eq:case2}
under the automorphism $\psi\in\Aut(\C^3)$ given by \eqref{eq:psi}. Obviously, the map 
\eqref{eq:case2} is an immersion if and only if $(h,g):M\setminus P\to\C^2$ is an immersion.
\end{proof}

The precise conditions for a Legendrian map $F=\Bscr(f,g)$ to be an immersion are
 more complicated.  By Lemma \ref{lem:gimmersion}, if $g$ is an immersion, then 
$\Bscr(f,g)$ is an immersion. See the discussion preceding Theorem \ref{th:Runge-for-immersions} 
for more information.

Let us look more closely at the formula \eqref{eq:Fhg}. The meromorphic $1$-form $hdg$ on $M$ is exact
if and only if $\int_C hdg=0$ for every closed curve $C$ in $M$ which does not contain any poles of $hdg$. 
There are two types of curves to consider:
those in a homology basis of $M$ (they can be chosen in the complement of the set of poles of 
of $hdg$), and small loops around the poles of $hdg$. The integral
of $hdg$ around a pole $a$  equals $2\pi \imath \, \Res_a (hdg)$. Let us record this observation.
  
%
%
\begin{proposition}\label{prop:exact}
A pair of meromorphic functions $(h,g)$ on a Riemann surface $M$ determines a Legendrian
immersion $F=\Fscr(h,g):M\to\CP^3$ \eqref{eq:Fhg} if and only if the following two conditions hold:
\begin{enumerate}[\rm (a)]
\item $\int_C hdg=0$ for every closed curve in a basis of the homology group $H_1(M,\Z)$, and 
\item $\Res_a (hdg)=0$ holds at every pole of $hdg$.
\end{enumerate}
If $a$ is a simple pole of $g$ or $h$, then condition {\rm (b)} is equivalent to
\begin{equation}\label{eq:Res}
	c_{-1}(h,a)c_1(g,a) - c_{-1}(g,a) c_{1}(h,a)=\Res_a (hdg)=0,
\end{equation}
where $c_k(h,a)$ denotes the coefficient of the term $(z-a)^k$ in a Laurent series representation
of $h$ at $a$ (so $c_{-1}(h,a)=\Res_a h$).
\end{proposition}

The situation is more complicated at poles of higher order. However, the case of first order
poles is a generic one in view of Proposition \ref{prop:all}.

\begin{proof}
It remains to show that \eqref{eq:Res} holds at a simple pole $a\in M$ of $h$ or $g$.
In a local holomorphic coordinate $z$ on $M$, with $z(a)=0$, we have that
\begin{eqnarray*}
	h(z) &=& \frac{c_{-1}(h)}{z}+ c_0(h) + c_1(h)z+\cdots,\\
	g(z) &=& \frac{c_{-1}(g)}{z}+ c_0(g) + c_1(g)z+\cdots,\\
	g'(z) &=& -\frac{c_{-1}(g)}{z^2}+ c_1(g)+\cdots, 
\end{eqnarray*}	
from which we easily infer that
\[
	\Res_0(hg') = c_{-1}(h)c_1(g) - c_{-1}(g)c_1(h).
\]
This gives \eqref{eq:Res} and completes the proof.
\end{proof}

\begin{remark}
In particular, if $a\in M$ is a simple pole of $h$ while $g$ is regular at $a$, we have 
\[
	\Res_a (hdg) = \Res_a (hg') = g'(a) \Res_a h.
\]
Similarly, if $a$ is a simple pole of $g$ while $h$ is regular at $a$, we have
\[
	\Res_a (hdg) = h'(a) \Res_a g.
\]
Assuming that $h$ and $g$ have only simple poles and no common pole,
we infer that the $1$-form $hdg$ has vanishing residues precisely when $h$ has a critical
point at each pole of $g$, and $g$ has a critical point at each pole of $h$.
\qed\end{remark}

%
%
\begin{remark}\label{rem:comparison}
An advantage of Bryant's formula \eqref{eq:Bfg} over \eqref{eq:Fhg} is that it
applies to any pair $(f,g)$ of meromorphic functions with $g$ nonconstant. A disadvantage is that the 
Legendrian curve $\Bscr(f,g)$ does not depend continuously on $(f,g)$ near a common critical point 
of $f$ and $g$ (see Remark \ref{rem:discontinuous}). 
This becomes a major drawback especially when trying to construct 
families of Legendrian curves depending continuously on parameters. 
A similar difficulty was encountered in \cite{ForstnericLarusson2018X} when 
studying holomorphic Legendrian curves in projectivised cotangent bundles. 
On the other hand, the Legendrian curve $\Fscr(h,g)$ \eqref{eq:Fhg} 
depends continuously on the pair $(h,g)$ for which $hdg$ is an exact $1$-form.
\qed\end{remark}

%
%
\section{Approximation and interpolation for Legendrian curves in $\CP^3$}   \label{sec:approx-interpol}

In this section we prove the Runge approximation theorem with interpolation at finitely many points for holomorphic 
Legendrian curves in $\CP^3$ (see Theorem \ref{th:Runge}) and holomorphic Legendrian immersions 
(see Theorem \ref{th:Runge-for-immersions}), both from compact and open Riemann surfaces. As a corollary, 
we obtain the interpolation theorem on a discrete set (see Corollary \ref{cor:interpolationCP3}).

We shall use the following version of Runge approximation theorem, proved by 
H.\ L.\ Royden \cite{Royden1967JAM} in 1967, which we state here for the reader's convenience.  
In this theorem, the given function is allowed to have poles on the set where the approximation takes place.

\begin{theorem}[Royden  \cite{Royden1967JAM}]
Let $M$ be a compact Riemann surface and $K\neq M$ be a compact subset of $M$.  
Let $E$ consist of one point in each connected component of $M\setminus K$, let $f$ be holomorphic on a
neighbourhood of $K$ except for finitely many poles in $K$, and $D$ be an effective divisor with support in $K$.  
Given $\epsilon>0$, there is a meromorphic function $F$ on $M$, holomorphic on $M\setminus E$ 
except at the poles of $f$, such that $(f-F)\geq D$ and $\lvert f-F\rvert<\epsilon$ on $K$.
\end{theorem}

The condition $(f-F)\geq D$ in the theorem simply means that $F$ agrees with $f$ to order $D(x)> 0$
at every point $x\in K$ of the finite support of the divisor $D$.

Here is our first approximation theorem.

%
%
\begin{theorem}   \label{th:Runge}
Let $M$ be a Riemann surface, open or compact, and let $K$ be a compact subset of $M$.  
Every holomorphic Legendrian map $\Phi$ from a neighbourhood of $K$ to $\CP^3$ can be 
approximated uniformly on $K$ by holomorphic Legendrian maps $M\to\CP^3$.  The approximants 
can be taken to agree with $\Phi$ to any finite order at each point of any finite subset of $K$.
\end{theorem}

We take a Riemann surface to be connected by definition, but the neighbourhood in the theorem 
need not be connected.

\begin{proof}
First we note that the compact case of the theorem implies the open case.  Indeed, if $M$ is open, we
exhaust $M$ by smoothly bounded compact domains containing $K$ and use induction, applying the compact 
case of the theorem to a compactification of each domain.  Hence, from now on, we assume that $M$ is 
compact.

Let $\Phi$ be a holomorphic Legendrian map from a neighbourhood $V$ of $K\neq M$ to $\CP^3$.  By 
Proposition \ref{prop:all}, we may assume that $\Phi=\Bscr(f,g)$, where $f$ and $g$ are meromorphic on 
$V$ and $g$ is not constant on any connected component of $V$.

Let $B$ be the finite subset of $K$ consisting of the poles of $f$, the poles of $g$, and the common critical 
points of $f$ and $g$ in $K$.  We use Royden's theorem to approximate $f$ and $g$ uniformly on a 
neighbourhood of $K$ by meromorphic functions $f_n$ and $g_n$ on $M$, respectively, such that the 
functions $\phi_n=f_n-f$ and $\psi_n=g_n-g$, which are holomorphic and go to zero uniformly on a 
neighbourhood of $K$, vanish at each point of $B$ to sufficiently high order $N$, independent of $n$, 
to be specified as the proof progresses.

We claim that if $N$ is sufficiently large, then the holomorphic Legendrian maps 
$\Bscr(f_n,g_n):M\to\CP^3$ converge to $\Bscr(f,g)$ uniformly on $K$ as $n\to\infty$.

Near a point $p$ of $K\setminus B$, with respect to a local coordinate $z$ centred at $p$,
\[ 
	\Bscr(f,g)=\big[g' : fg'-\tfrac 1 2 f'g : gg' : \tfrac 1 2 f'\big]. 
\]
On a neighbourhood $U$ of $p$ with $U\cap B=\varnothing$, $f_n\to f$ and $g_n\to g$ uniformly, 
these functions are holomorphic, and the same holds for their derivatives, so
\[
	\big(g_n', f_ng_n'-\tfrac 1 2 f_n'g_n, g_ng_n', \tfrac 1 2 f_n'\big) 
	\lra \big(g', fg'-\tfrac 1 2 f'g, gg', \tfrac 1 2 f'\big) 
\]
uniformly on $U$ as $n\to\infty$.  Also, $(g', fg'-\tfrac 1 2 f'g, gg', \tfrac 1 2 f')\neq (0,0,0,0)$ at every point of $U$, 
so $\Bscr(f_n,g_n)\to \Bscr(f,g)$ uniformly on $U$.

Next, let $p\in B$.  Then the lowest order $m\in\mathbb Z$ at $p$ of the components $g'$, $fg'-\tfrac 1 2 f'g$, 
$gg'$, $\tfrac 1 2 f'$ of $\Bscr(f,g)$ is not zero.  If $N$ is large enough, then a component of $\Bscr(f,g)$ 
of order $m$ corresponds to a component of $\Bscr(f_n,g_n)$ of lowest order, and that lowest order is also $m$.
If we now multiply the components by $z^{-m}$, then we are in the same situation as before and need to 
show that 
\begin{multline*} 
	\big(z^{-m}g_n', z^{-m}(f_ng_n'-\tfrac 1 2 f_n'g_n), z^{-m}g_ng_n', 
	\tfrac 1 2 z^{-m}f_n'\big) \\ \lra \big(z^{-m}g', z^{-m}(fg'-\tfrac 1 2 f'g), z^{-m}gg', \tfrac 1 2 z^{-m}f'\big) \end{multline*}
uniformly near $p$ as $n\to\infty$.  Note that each difference
\begin{multline*} 
	z^{-m}g_n'-z^{-m}g', \quad z^{-m}(f_ng_n'-\tfrac 1 2 f_n'g_n) - z^{-m}(fg'-\tfrac 1 2 f'g), \\ 
	z^{-m}g_ng_n' - z^{-m}gg', \quad \tfrac 1 2 z^{-m}f_n' - \tfrac 1 2 z^{-m}f' 
\end{multline*}
is a sum of terms of the form $z^{-m}$ times one of the functions
\[ 
	\phi_n',\ \psi_n',\  \phi_n\psi_n',\ \phi_n'\psi_n,\ 
	\psi_n\psi_n',\ f'\psi_n,\ f\psi_n',\ g'\phi_n,\ g\phi_n',\ g'\psi_n,\ g\psi_n' 
\]
(perhaps with a factor of $\tfrac 1 2$).  By the maximum principle, $z^{-N}\phi_n$ and $z^{-N}\psi_n$ 
go to zero uniformly near $p$.  Likewise, $z^{-N+1}\phi_n'$ and $z^{-N+1}\psi_n'$ go to zero uniformly 
near $p$.  Hence, if $N$ is big enough, all those differences go to zero uniformly near 
every point $p$ in the finite set $B$.

Jet interpolation can be achieved by taking $N$ large enough and, if necessary, 
adding finitely many points to $B$.
\end{proof}

Next we adapt Theorem \ref{th:Runge} to immersions.  First we need to determine those meromorphic functions $f$ and $g$ for which $\Bscr(f,g)$ is an immersion.

First, if $f$ is constant, then $\Bscr(f,g)=[dg: fdg: gdg:0] = [1:f:g:0]$ is an immersion if and only if $g$ is an immersion.  Now suppose that $f$ is not constant.  In suitable local coordinates centred at a point $p$ in $M$ and at the point $g(p)$ in $\CP^1$, write $g(x)=x^b$, $b\neq 0$, and $f(x)=x^ah(x)$, where $h$ is holomorphic near $p$ and $h(0)\neq 0$.  If $a=0$, then 
\[ 
	\Bscr(f,g) = \left[bx^{b-1} : bx^{b-1}f(x)-\tfrac 1 2 x^bf'(x) : bx^{2b-1} :\tfrac 1 2 f'(x)\right]. 
\]
The orders of the components at $p$ are
\[ 
	\left[b-1 : b-1: 2b-1:\operatorname{ord}_p f'\geq 0 \right]. 
\]
If $a\neq 0$, then
\begin{eqnarray*}
	\Bscr(f,g) &=& \big[ bx^{b-1} : bx^{b-1}x^ah(x)-\tfrac 1 2 x^b(ax^{a-1}h(x)+x^ah'(x)) : \\
	&&  \qquad \qquad\qquad\qquad  bx^{2b-1} : \tfrac 1 2 (ax^{a-1}h(x)+x^ah'(x)) \big], 
\end{eqnarray*}
so the orders of the components at $p$ are
\[ 
	[b-1 : c : 2b-1 : a-1], 
\]
where
\[ 
	c=\left\{ 
	\begin{array}{cl} a+b-1 & \textrm{if $a\neq 2b$,} \\ a+b+\operatorname{ord}_p h' & \textrm{if $a=2b$.} 		\end{array} \right. 
\]
Now $\Bscr(f,g)$ is regular at $p$ if and only if the smallest and the second smallest order differ by 1.  It is easily checked that this condition is satisfied when $g$ is regular at $p$, that is, when $b=\pm 1$.  Indeed, for $b=1$, the orders are
\[ 
	\begin{array}{cl} {[0:0:1: \,\geq 0]} & \textrm{if $a=0$,} \\ {[0:\, \geq 3 :1:1]} & \textrm{if $a=2$,} 
	\\ {[0:a:1:a-1]} & \textrm{if $a\neq 0, 2$,} 
	\end{array}
\]
and for $b=-1$, the orders are
\[ 
	\begin{array}{cl} {[-2:-2:-3: \,\geq 0]} & \textrm{if $a=0$,} \\
	{[-2: \,\geq -3:-3:-3]} & \textrm{if $a=-2$,} \\ 
	{[-2:a-2:-3:a-1]} & \textrm{if $a\neq 0, -2$.}
	\end{array} 
\]
Let us record this observation.

\begin{lemma}\label{lem:gimmersion}
If $g$ is an immersion, then $\Bscr(f,g)$ is an immersion.
\end{lemma}

We see that when $g$ is critical at $p$, that is, $b\geq 2$ or $b\leq -2$, then $\Bscr(f,g)$ is regular at $p$ if, for example, $a=b-1$.  On the other hand, regularity of $f$ may not be enough to ensure regularity of $\Bscr(f,g)$, for example when $a=1$ and $b=3$.

In fact, we see that if $g$ is critical at $p$, then $\Bscr(f,g)$ is regular at $p$ if and only if the degrees of the first two terms in the Laurent series of $f$ at $p$ belong to a certain set of admissible pairs of integers that only depends on the order of $g$ at $p$ (and that is quite complicated to describe explicitly).

%
%
\begin{theorem}   \label{th:Runge-for-immersions}
Let $M$ be a Riemann surface, open or compact, and let $K$ be a compact subset of $M$.  
Every holomorphic Legendrian immersion $\Phi$ from a neighbourhood of $K$ to $\CP^3$ can be uniformly 
approximated on $K$ by holomorphic Legendrian immersions $M\to\CP^3$.  The approximants can be 
taken to agree with $\Phi$ to any given finite order at each point of any given finite subset of $K$.
\end{theorem}

\begin{remark}\label{rem:Runge-complete}
We shall see in Corollary \ref{cor:completeLeg} (ii) that the approximating immersion 
$M\to\CP^3$ in Theorem \ref{th:Runge-for-immersions} can always be chosen {\em complete}, 
i.e., such that the pullback of any Riemannian metric on $\CP^3$ by the immersion is 
a complete metric on $M$. 
\qed\end{remark}

\begin{proof}[Proof of Theorem \ref{th:Runge-for-immersions}]
As in the proof of Theorem \ref{th:Runge}, it suffices to take $M$ to be compact.  By Theorem \ref{th:Runge}, a 
Legendrian immersion $\phi$ from a neighbourhood $U$ of $K$ to $\CP^3$ can be approximated on $K$ by a 
Legendrian map $\Bscr(f_0,g):M\to\CP^3$.  If we approximate sufficiently well on a compact neighbourhood of $K$ 
in $U$, then $\Bscr(f_0,g)$ will be an immersion on a neighbourhood $V$ of $K$.  On $M\setminus V$, $g$ has
finitely many critical points, at which $\Bscr(f_0,g)$ may not be regular.

Let $h$ be a meromorphic function on a neighbourhood of the disjoint union $L$ of a compact neighbourhood $K'$ 
of $K$ and closed coordinate discs centred at each of the critical points of $g$ in $M\setminus K$, such that 
$h=f_0$ near $K'$, and $h$ has order $b-1$ at each critical point of $g$ of order $b$.  As in the proof of Theorem
\ref{th:Runge}, we can use Royden's theorem to approximate $h$ uniformly on $L$ by a meromorphic function $f$
 on $M$ such that: 
\begin{itemize}
\item $\Bscr(f,g)$ is as close as we wish to $\Bscr(f_0,g)$ on $K'$, so in particular, 
$\Bscr(f,g)$ is an immersion on a neighbourhood of $K$,
\item at each critical point of $g$ in $M\setminus K$ of order $b$, $f$ has order $b-1$, 
so $\Bscr(f,g)$ is regular there.
\end{itemize}
Then $\Bscr(f,g):M\to\CP^3$ is a Legendrian immersion that uniformly approximates $\phi$ 
on $K$ as closely as desired.  Finally, jet interpolation can be incorporated using Theorem \ref{th:Runge}.
\end{proof}

%
%
\begin{corollary}\label{cor:interpolationCP3} 
Let $E$ be a closed discrete subset of a Riemann surface $M$.  Every map 
$E\to\CP^3$ can be extended to a holomorphic Legendrian immersion $M\to\CP^3$.
\end{corollary}

\begin{proof}
For a compact Riemann surface $M$ this is a corollary of Theorem \ref{th:Runge-for-immersions},
applied to $K$ being the union of small mutually disjoint discs around the points of $E$.
For an open Riemann surface we apply Theorem \ref{th:Runge-for-immersions} inductively,
interpolating at more and more points of the given discrete set as we go.
\end{proof}

%
%
\begin{corollary}\label{cor:Runge-embeddings}
Let $M$ be an open Riemann surface.  Every holomorphic Legendrian immersion $M\to\CP^3$ can be approximated, uniformly on compact subsets of $M$, by holomorphic Legendrian embeddings $M\hra\CP^3$.
\end{corollary}

\begin{proof}
Let $\phi:M\to\CP^3$ be a holomorphic Legendrian immersion.  Choose an 
exhaustion $K_1\subset K_2\subset\cdots$ of $M$ by compact subsets without holes
so that each $K_j$ is contained in the interior of the next set $K_{j+1}$.  
By \cite[Theorem 1.2]{AlarconForstneric2019IMRN}, $\phi$ can be approximated arbitrarily closely, 
uniformly on $K_3^\circ$, by a holomorphic Legendrian embedding $\psi:K_3^\circ\to\CP^3$.  
By Theorem \ref{th:Runge-for-immersions}, $\psi$ can be approximated uniformly on 
$K_2$ by a holomorphic Legendrian immersion $\phi_1:M\to\CP^3$.  If the approximation is close enough 
then $\phi_1$ is injective on $K_1$. Repeating the same argument, $\phi_1$ can be approximated arbitrarily 
closely on $K_3$ by a holomorphic Legendrian immersion $\phi_2:M\to\CP^3$ that is an embedding on
$K_2$.  Continuing in this way with sufficiently close approximations and passing to the limit, 
the corollary is proved.
\end{proof}

%
%
\begin{problem}
Let $M$ be a compact Riemann surface. Is it possible to approximate every 
holomorphic Legendrian immersion $M\to\CP^3$ by a holomorphic Legendrian embedding?
\end{problem}

In this connection, Bryant did prove in \cite[Theorem G]{Bryant1982JDG} that every 
compact Riemann surface admits a holomorphic Legendrian embedding into $\CP^3$,
but his proof does not seem to provide an answer to the above problem, and we could not find one either.

\begin{remark}\label{rem:discontinuous}
We have a bijection $(f,g)\mapsto\Bscr(f,g)$ from the space of pairs $(f,g)$ of meromorphic functions on $M$ 
with $g$ nonconstant to the space of holomorphic Legendrian maps $M\to\CP^3$ whose image does not lie in 
a plane of the form $[z_0:z_2] = \text{\rm constant}$.  As noted in Remark \ref{rem:comparison}, this
bijection is not continuous.  (The analogous phenomenon for projectivised cotangent bundles was observed in 
\cite{ForstnericLarusson2018X}.)  Take, for example, $M=\C$, $f(x)=x^2$, and $g_\epsilon(x)=(x+\epsilon)^2$, 
$\epsilon\in\C$.  Then
\[ 	
	\Bscr(f,g_\epsilon)(x) = \left[x+\epsilon: \tfrac 1 2 x(x^2-\epsilon^2): (x+\epsilon)^3: x\right] 
\]
and in particular, $\Bscr(f,g_0)(x)=\left[1:\tfrac 1 2 x^2: x^2: 1\right]$, so 
$\Bscr(f,g_\epsilon)(0)=[1:0:\epsilon^2:0]$ for $\epsilon\neq 0$, but $\Bscr(f,g_0)(0)=[1:0:0:1]$.

The inverse bijection $\Bscr^{-1}$, however, is continuous.  Indeed, we can retrieve $g$ from $\Bscr(f,g)$ by postcomposing by the meromorphic function 
\[ 
	\psi:[z_0:z_1:z_2:z_3] \,\longmapsto\, \dfrac{z_2}{z_0}, 
\]
and retrieve $f$ by postcomposing by
\[ 
	\phi:[z_0:z_1:z_2:z_3] \, \longmapsto\, \dfrac{z_0 z_1 + z_2 z_3}{z_0^2}, 
\]
so $\Bscr^{-1}(h)=(\phi\circ h, \psi\circ h)$ for a holomorphic Legendrian map $h:M\to\CP^3$ whose image does not lie in a plane of the form $[z_0:z_2] = \text{\rm constant}$.  To see that $\Bscr^{-1}$ is continuous, note that the image of $h$ will not lie in the indeterminacy locus of either $\phi$ or $\psi$, since both loci lie in the plane where $z_0=0$.

As shown in the proof of Theorem \ref{th:Runge}, despite the failure of continuity of $\Bscr$, if $(f_n,g_n) \to (f,g)$ uniformly on a neighbourhood of a compact subset $K$ of a Riemann surface and the functions $f_n-f$ and $g_n-g$, which are holomorphic and go to zero uniformly on a neighbourhood of $K$, vanish to sufficiently high order, independent of $n$, at each pole of $f$, pole of $g$, and common critical point of $f$ and $g$ in $K$,  then the holomorphic Legendrian maps $\Bscr(f_n,g_n)$ converge to $\Bscr(f,g)$ uniformly on $K$ as $n\to\infty$.
\qed\end{remark}

%
%

\section{The space of Legendrian immersions $M\to \CP^3$ is path connected}   \label{sec:connected}

In this section we prove the following result. 

\begin{theorem} \label{th:connected}
The space of holomorphic Legendrian immersions from an open Riemann surface to $\CP^3$ is path connected 
in the compact-open topology. 
\end{theorem}

From a purely technical viewpoint, this may be the most difficult result in the paper. 
The fact that homogeneous coordinates on $\CP^3$ can be chosen such that the meromorphic functions, 
defining a given immersed holomorphic Legendrian curve, have only simple poles (see Proposition \ref{prop:all}) 
is essential in our proof. 

We shall need the following parametric version of Weierstrass's interpolation theorem 
for finitely many points in an open Riemann surface. 
A similar result holds for a variable family of infinite discrete sets of points, 
but this simple version suffices for our present application. 

%
%
\begin{lemma}\label{lem:Weierstrass}
Let $M$ be an open Riemann surface. Given maps $a_j:[0,1]\to M$, $j=1,\ldots,k$,
of class $\Cscr^r$ for some $r\in\{0,1,\ldots,\infty,\omega\}$
and integers $n_1,\ldots,n_k\in\N$, there is a path of holomorphic functions $f_t\in \Oscr(M)$
with $\Cscr^r$ dependence on $t$  such that for every $t\in [0,1]$ and
$j=1,\ldots,k$, the function $f_t$ vanishes to order $n_j$ at $a_j(t)$ and has no other zeros.
 \end{lemma}

\begin{proof}
It suffices to prove the result for $k=1$ and $n_1=1$; the general case is 
then obtained by taking for each $j=1,\ldots,k$, a path of functions $f_{j,t}$ with 
a simple zero at $a_j(t)$ and no other zeros, and letting $f_t=\prod_{j=1}^k f_{j,t}^{n_j}$.

Hence, let $a:[0,1]\to M$ be a $\Cscr^r$ function. Assume first that $a$ is
real analytic. Then, $a$ complexifies to a holomorphic map from an open simply 
connected neighbourhood $D\subset \C$ of the interval $[0,1]\subset \R\subset \C$ to $M$.
Its graph $\Sigma=\{(z,a(z)):z\in D\} \subset D\times M$ is a smooth complex hypersurface
which defines a divisor on the Stein surface $D\times M$. Since we clearly have 
$H^2(D\times M,\Z)=0$, K.\ Oka's solution of the second Cousin problem \cite{Oka1939} 
shows that this divisor is defined by a holomorphic function $f\in \Oscr(D\times M)$ 
which vanishes to order $1$ on $\Sigma$ and has no other zeros. The function
$f_t=f(t,\cdotp)\in \Oscr(M)$ then has a simple zero at $a(t)$ and
no other zeros, and its dependence on $t\in [0,1]$ is real analytic.

If $a$ is of class $\Cscr^r$ but not real analytic, we proceed as follows.
Choose a nowhere vanishing holomorphic vector field $V$ on $M$ and a relatively
compact Runge domain $M_0\Subset M$ such that $a(t)\in M_0$ for all $t\in [0,1]$.
There is $\epsilon>0$ such that the flow $\phi_s(x)$ of $V$ exists 
for any initial point $\phi_0(x)=x\in \overline M_0$ and for all $s\in \C$ with $|s|<\epsilon$, 
and $s\mapsto \phi_s(x)$ maps the disc  $|s|<\epsilon$ biholomorphically onto a neighbourhood
$U(x)\subset M$ of $x$. The diameter of these neighbourhoods in 
a fixed metric on $M$ is uniformly bounded from below for $x\in \overline M_0$.
Hence, approximating $a:[0,1]\to M_0$ sufficiently closely by a real analytic
map $\tilde a:[0,1]\to M$, we have that $\tilde a(t)=\phi_{s(t)}(a(t))$ for a unique
$\Cscr^r$ function $s=s(t)$ with $|s(t)|<\epsilon$ for all $t\in [0,1]$. 
If $\tilde f_t\in\Oscr(M)$ is a real analytic path of functions with simple zeros at
$\tilde a(t)$  for $t\in[0,1]$, then $f_t=\tilde f_t\circ \phi_{s(t)} \in \Oscr(M_0)$ is a $\Cscr^r$ 
path of functions with simple zeros at $a(t)$.

To complete the proof, we approximate the path $f_t$ by a path of
holomorphic functions on $M$ without creating additional zeros.
This is done inductively, exhausting $M$ by an increasing sequence
of Runge domains $M_0\subset M_1\subset \cdots \subset \bigcup_{j=1}^\infty M_j=M$ 
and constructing a sequence of $\Cscr^r$ paths $f_{j,t}\in \Oscr(M_j)$ $(j\in \Z_+)$ 
having simple zeros at $a(t)$ and converging to a $\Cscr^r$ path $f_t\in\Oscr(M)$ 
with the same property. Note that $f_{j+1,t}=f_{j,t}g_{j,t}$ on $M_j$, 
where $g_{j,t}$ is a $\Cscr^r$ path in $\Oscr(M_j,\C^*)$. By the parametric Oka
theorem for maps to $\C^*$, we can approximate the path $g_{j,t}$ 
by a $\Cscr^r$ path $\tilde g_{j,t}\in\Oscr(M_{j+1},\C^*)$.
Replacing $f_{j+1,t}$ by $f_{j,t} \tilde g_{j,t}$ gives a $\Cscr^r$ path in $\Oscr(M_{j+1},\C^*)$
which approximates $f_{j,t}$ as closely as desired on a chosen compact subset of $M_j$. 
Assuming that the approximations are close enough, the sequence $f_{j,t}$ converges as $j\to\infty$
to a $\Cscr^r$ path $f_t\in\Oscr(M)$ solving the problem.
\end{proof}

%
%
\begin{proof}[Proof of Theorem \ref{th:connected}]
Given a pair of Legendrian immersions $F_0,F_1:M\to\CP^3$, we must find a 
path of Legendrian immersions $F_t:M\to\CP^3$, $t\in[0,1]$, connecting $F_0$ to $F_1$. 

Choose a hyperplane $H\subset \CP^3$ such that $F_0$ and $F_1$ are transverse to $H$.
(Most hyperplanes satisfy this condition by Bertini's theorem, cf.\ \cite{Kleiman1974}.)
By Proposition \ref{prop:all}, there are homogeneous coordinates $[z_0:z_1:z_2:z_3]$ on $\CP^3$,
with $H=\{z_0=0\}$, such that $F_j=\Fscr(h_j,g_j)$ $(j=0,1)$ where 
$h_j$, $g_j$ are meromorphic functions on $M$ with only simple poles satisfying
conditions \eqref{eq:Res}. To prove the theorem, we shall find 
a path $(h_t,g_t)$ $(t\in [0,1])$ of pairs of meromorphic functions on $M$ 
connecting $(h_0,g_0)$ to $(h_1,g_1)$ and satisfying the following conditions for every $t\in [0,1]$.
(By the assumptions, these conditions hold for $t=0,1$.)
\begin{enumerate}[\rm (i)]
\item $h_t$ and $g_t$ have only simple poles and the relations \eqref{eq:Res} hold.
\item $\int_C h_t dg_t=0$ for every closed curve $C$ in a homology basis of $M$.
\item Let $P_t=P(h_t)\cup P(g_t)\subset M$ denote the union of the polar loci of $h_t$ and $g_t$. 
Then the map $(h_t,g_t):M\setminus P_t\to\C^2$ is an immersion.
\end{enumerate}
In light of (i), condition (iii) is clearly equivalent to 
\begin{enumerate}[\rm (iv)]
\item $(h_t,g_t):M\to(\CP^1)^2$ is an immersion.
\end{enumerate}
The path of holomorphic Legendrian immersions $F_t=\Fscr(h_t,g_t):M\to \CP^3$, $t\in [0,1]$, 
then satisfies the conclusion of the theorem. 

We shall do this by inductively approximating a path $(h_t,g_t)$ satisfying 
the stated conditions on some connected Runge domain $D\Subset M$ 
by a path satisfying the same conditions  on a bigger Runge domain $D'\Subset M$. 
To be precise, we exhaust $M$ by an increasing sequence 
\[
	K_1\subset K_2\subset \cdots\subset \bigcup_{j=1}^\infty K_j = M
\]
of compact smoothly bounded domains without holes such that $K_j\subset K_{j+1}^\circ$ 
for every $j\in\N$. (The set $K_1$ may be chosen as big as desired.) At every stage 
we shall approximate a given family of solutions $(h^j_t,g^j_t)$ on a neighbourhood of $K_j$ by 
one on a neighbourhood of $K_{j+1}$ which has the same jets at a prescribed finite family of points in $K_j$.
Furthermore, we will ensure that $(h^j_t,g^j_t)$ agrees with the given pair $(h_t,g_t)$ 
for $t=0$ and $t=1$.

Since we shall be using partitions of unity on $[0,1]$, we must give ourselves some  
freedom at the endpoints. To this end, choose a small number $0<r_1<1/2$ and define 
\[
	(h_t,g_t)=\begin{cases} (h_0,g_0) & \text{ for}\ t\in[0,r_1], \\
  					      (h_1,g_1) &  \text{ for}\ t\in[1-r_1,1].
			\end{cases}
\]
For $t\in [0,r_1]\cup [1-r_1,1]$ let $A_t,B_t$, $C_t$ denote closed discrete subsets of $M$ such that
\begin{enumerate}[\rm (a)]
\item $A_t$ is the set of poles of $h_t$ which are not poles of $g_t$.
\item $B_t$ is the set of poles of $g_t$ which are not poles of $h_t$.
\item $C_t$ is the set of common poles of $g_t$ and $h_t$.
\end{enumerate}
Thus, the set 
\[	
	P_t:=A_t\cup B_t\cup C_t \subset M
\]
is the union of polar loci of $h_t$ and $g_t$.
(The reason for specifically distinguishing points in $C_t$ will become apparent shortly.)
These sets do not depend on $t$ in the indicated pair of intervals, but they will become $t$-dependent
in subsequent steps when extending them to bigger sets of parameter values $t\in[0,1]$.

In the initial stage of the induction process we shall construct a path $(h_t,g_t)$ of pairs of 
meromorphic functions on an open neighbourhood $U=U_1$ of $K_1$ in $M$
such that conditions (a)--(c) above hold for all $t\in [0,1]$ at the points in $P_t\cap U$. 
This will be done in four steps. The neighbourhood $U$ may shrink around 
$K_1$ at every step without saying so each time.

Fix once and for all a holomorphic immersion $\zeta:M\to \C$
(as provided by the theorem of R.\ Gunning and R.\ Narasimhan \cite{GunningNarasimhan1967}),
so $\zeta$ provides a local holomorphic coordinate at every point of $M$. 
We can express any meromorphic function $h$ in a neighbourhood of a point $p\in M$ 
by a Laurent series in the local holomorphic coordinate $z=\zeta-\zeta(p)$. 
We denote by $c_k(h,p)$ the $k$-th coefficient of $h$ at $p$ in this series.
Note that these coefficients are already defined for our functions $h_t,g_t$ near $t=0,1$, 
they vanish for $k<-1$ since the functions have only simple poles, 
and they satisfy conditions \eqref{eq:Res}.

%
%
\smallskip\noindent
{\em Step 1.} 
Choose a connected open neighbourhood $D_1$ of $K_1$ such that 
$bD_1\cap P_0=bD_1\cap P_1=\varnothing$. For $t\in [0,r_1]\cup [1-r_1,1]$, let 
\[	
	A^1_t= A_t\cap D_1,\quad B^1_t= B_t\cap D_1,\quad C^1_t= C_t\cap D_1.
\]
We now extend each of these sets to all parameter values $t\in[0,1]$ 
(possibly adding more points to the already given sets) as follows.
We connect a point $a\in A^1_0$ to a point $a'\in A^1_1$ by a smooth path $a(t)\in D_1$, $t\in[0,1]$,
so that $a(t)=a$ for $t\in [0,r_1]$ and $a(t)=a'$ for $t\in [1-r_1,1]$, ensuring
that paths of this kind with distinct initial points remain distinct at all $t\in [0,1]$. This is possible
for all points in $A^1_0$ and $A^1_1$ if and only if they have the same cardinality. 
If $A^1_0$ has more points than $A^1_1$, we choose a path $a(t)$ starting at 
a point in $A^1_0$ without a matching pair in $A^1_1$ such that $a(t)$ exits $D_1$ at some time $t_0\in (0,1)$, 
and we define it in an arbitrary way (but staying in $M\setminus D_1$) for the remaining values 
$t\in (t_0,1]$. If on the other hand $A^1_1$ has more points than $A^1_0$, we do the same for points in 
$A^1_1$ without matching pairs in $A^1_0$, with $t$ now running from $t=1$ back to $t=0$. 
(The parts of these paths lying in $M\setminus D_1$ will be redefined in the next stage
of the induction process.) Let $A^1_t\subset M$, $t\in [0,1]$ denote the finite set of points 
obtained in this way. Note that the cardinality of $A^1_t$ for each $t$ equals the bigger of the 
cardinalities of $A_0\cap D_1$ and $A_1\cap D_1$, and in the process we may have added 
more points to these sets.  

We repeat the same procedure with the points in the families $B$ and $C$, making sure that 
the resulting sets $A^1_t,B^1_t,C^1_t$ are pairwise disjoint for all $t\in[0,1]$. Let 
\begin{equation}\label{eq:P1t}
	P^1_t:=A^1_t\cup B^1_t\cup C^1_t.
\end{equation}
By construction, the points in $P^1_t$ vary smoothly with $t$ and their number
does not depend on $t$. Hence,  by Lemma \ref{lem:Weierstrass} there is a smooth path of 
holomorphic functions $f_t\in \Oscr(M)$, $t\in [0,1]$, vanishing to order $2$ at each of the points 
in $P^1_{t}$ and nowhere else.

%
%
\smallskip\noindent
{\em Step 2.} 
Let $D_1$ be the neighbourhood of $K_1$ as in step 1.
We extend the meromorphic jets of $(h_t,g_t)$ containing terms of orders $-1,0,1$ 
in their Laurent series expansion, which are already defined at points in $P^1_t\cap D_1$ for 
$t\in [0,r_1]\cup [1-r_1,1]$, to a smooth path of jets defined at all points of $P^1_t\cap D_1$, 
$t\in[0,1]$, such that conditions \eqref{eq:Res} hold. (Recall that these conditions ensure 
the existence of local meromorphic primitives of $h_tdg_t$ at the points in $P^1_t\cap D_1$.) 

Let us explain the details. We are interested in jets of the form
\begin{eqnarray}\label{eq:xih}
	\xi^h_p(\zeta) &=& c_{-1}^h(p) (\zeta-\zeta(p))^{-1} + c_{0}^h(p) + c_{1}^h(p)(\zeta-\zeta(p)),\\
	\xi^g_p(\zeta) &=& c_{-1}^g(p) (\zeta-\zeta(p))^{-1} + c_{0}^g(p) + c_{1}^g(p)(\zeta-\zeta(p)).
\end{eqnarray}
For $p\in P^1_t\cap D_1$ $(t\in [0,r_1]\cup [1-r_1,1])$ and $j=-1,0,1$, we set
\[
	c_{j}^h(p)=c_{j}(h_t,p),\quad  c_{j}^g(p)=c_{j}(g_t,p),
\]
where $h_t,g_t$ are the initially given meromorphic functions.
It is elementary to extend the coefficients $c_{j}^h,c_{j}^g$ to smooth functions
\[
	c_{j}^h, c_{j}^g : P^1_t\cap D_1  \to \C,\quad t\in [0,1],\ j=-1,0,1,
\]
satisfying the following conditions.
\begin{itemize}
\item At points $p\in A^1_t\cap D_1$ we have $c_{-1}^h(p)\ne 0$ and $c_{1}^g(p)=0$.
\item At points $p\in B^1_t\cap D_1$ we have $c_{-1}^g(p)\ne 0$ and $c_{1}^h(p)=0$.
\item At points $p\in C^1_t\cap D_1$ we have $c_{-1}^h(p)\ne 0$, $c_{-1}^g(p)\ne 0$, and 
\[	
	c_{-1}^h(p) c_{1}^g(p)-c_{-1}^g(p) c_{1}^h(p)=0.
\]
\end{itemize}
These are precisely conditions \eqref{eq:Res}. There are no conditions on $c_0^h$ and $c_0^g$.

\begin{remark}
The above conditions on points $p\in C_t$ allow nonzero values
of $c_{1}^h(p)$ and $c_{1}^g(p)$, while those for points $p\in A_t$ force $c_{1}^g(p)=0$,
and those for points $p\in B_t$ force $c_{1}^h(p)=0$. Hence, if a common pole of 
$h_t$ and $g_t$ split into a pair of distinct poles of these functions
for nearby values of $t$, the required conditions could not always be satisfied in a continuous way.
For this reason, these three types of poles must remain distinct for all values of $t$.
\qed\end{remark}

%
%
\noindent
{\em Step 3.}
We shall find smooth paths $h_t$ and $g_t$ of meromorphic functions on a neighbourhood
$U\subset D_1$ of $K_1$ having the jets constructed in step 2 at the points of $P^1_t\cap U$.
(It will suffice to use paths of class $\Cscr^1$ in the variable $t\in[0,1]$.)
In particular, these functions will agree with the already given ones for $t$ near $0$ and $1$, and they 
will satisfy the following conditions for every $p\in P_t\cap U$ and $t\in [0,1]$:
\[
	c_{\pm 1}^h(p)=c_{\pm 1}(h_t,p),\quad\ c_{\pm 1}^g(p)=c_{\pm 1}(g_t,p).
\] 
These are Mittag-Leffler interpolation conditions at a variable family of points in the open 
Riemann surface $M$. This problem can be solved by using the $\dibar$-equation
together with Lemma \ref{lem:Weierstrass}. An important point is that a convex combination 
of solutions is again a solution, a fact which allows for the use of partitions of unity in the $t$-variable.
A detail that one must pay attention to is that the number of points in the sets
$P^1_t\cap D_1$ may vary with $t$. We now explain how to do this.

Fix a point $t_0\in (0,1)$. We shall first solve the problem locally for $t$ near $t_0$.
(For $t_0=0,1$, we take the already given functions defined on all of $M$.) 
Choose a domain $D'_1$ with $K_1\subset D'_1\subset D_1$
such that $P^1_{t_0} \cap bD'_1=\varnothing$ (see \eqref{eq:P1t}). Then, there is an open
interval $I_{t_0}\subset [0,1]$ around $t_0$ such that  $P^1_{t} \cap bD'_1=\varnothing$
for all $t\in I_{t_0}$. Hence, the number of points in the set
\[
	P^1_{t} \cap D'_1 = \{p_1(t),\ldots,p_k(t)\} 
\]
is independent of $t\in I_{t_0}$. Choose small pairwise disjoint coordinate neighbourhoods 
$U_j\subset D'_1$ of the points $p_j(t_0)$ for $j=1,\ldots, k$, and for each $j$ choose a 
smooth function $\chi_j:M\to[0,1]$ which equals $1$ on a neighbourhood $V_j\Subset U_j$ 
of $p_j(t_0)$ and has support contained in $U_j$. Shrinking the interval $I_{t_0}$ around $t_0$,
we may assume that $p_j(t)\in V_j$ for all $t\in \overline I_{t_0}$ and $j=1,\ldots, k$.
Recall that the jet $\xi^h_p$ is given by \eqref{eq:xih}. We define 
\[
	\tilde h_t(x) = \sum_{j=1}^k \chi_j(x) \, \xi^h_{p_j(t)}(\zeta(x)) \quad \text{for}\ x\in M.
\]
The same expression, with $\xi^h$ replaced by $\xi^g$, defines $\tilde g_t$. 
Note that $\tilde h_t$ is a smooth function on $M\setminus P^1_t$ 
whose restriction to $V_j$ agrees with $\xi^h_{p_j(t)}$ for every $j=1,\ldots,k$;
the analogous statement holds for $\tilde g_t$.  Since the number of points 
$p_j(t)\in P^1_{t} \cap D'_1$ is independent of $t\in I_{t_0}$, Lemma \ref{lem:Weierstrass} furnishes a path
of holomorphic functions $f_t\in \Oscr(M)$ with $t\in I_{t_0}$, vanishing to order $2$ at these points 
and nowhere else.  We look for the desired path $(h_t,g_t)$ with $t\in I_{t_0}$, in the form
\[
	h_t=\tilde h_t - f_t  \mu_t,\qquad  g_t=\tilde g_t - f_t \nu_t,
\]
where $\mu_t,\nu_t$ are paths of smooth functions to be found.  
The choice of $f_t$ implies that $(h_t,g_t)$ has the same jet with coefficients
$-1,0,1$ as $(\tilde h_t,\tilde g_t)$ at the points in $P^1_{t} \cap D'_1$, and hence conditions 
\eqref{eq:Res} still hold for $(h_t,g_t)$.

The condition that $h_t$ is holomorphic (except at the poles in $P^1_t\cap D'_1$) is  
\[
	0=\dibar h_t = \dibar\, \tilde h_t - f_t\, \dibar \mu_t  
	\ \Longleftrightarrow\ \dibar \mu_t =\frac{\dibar\, \tilde h_t}{f_t}=:\alpha_t.
\]
Note that $\dibar \,\tilde h_t$ vanishes in $V^j_t$ for each $j=1,\ldots,k$, and 
since $f_t$ has zeros only on $P^1_t\cap D'_1$, $\alpha_t$ is a 
smooth $(0,1)$-form on $M$ depending smoothly on $t\in I_{t_0}$.  
Hence, the equation $\dibar \mu_t =\alpha_t$ has a solution depending smoothly on 
$t\in I_{t_0}$, and we get a desired path of functions $h_t$ as above. The same procedure
applies to $g_t$. 

%
%
\begin{remark}[On the parametric $\dibar$-equation]
There are several approaches in the literature to solving the nonhomogeneous
$\dibar$-equation by bounded linear operators, which therefore also apply in the parametric case.
In the simple case at hand we have a $1$-parameter family of $\dibar$-equations on a domain in an open
Riemann surface. In this case, a solution operator -- a Cauchy-Green-type operator 
similar to the one in the complex plane -- has already been constructed by 
H.\ Behnke and K.\ Stein in 1949; see \cite{BehnkeStein1949}. A discussion 
of this topic can be found in \cite[Sect.\  2]{FornaessForstnericWold2018};
see in particular Remark 1 there.
\qed\end{remark}

It remains to combine the partial solutions, obtained in this way 
on parameter subintervals $I_{t_0}\subset [0,1]$, into a solution $(h^1_t,g^1_t)$ $(t\in [0,1])$ 
over a neighbourhood $U$ of $K_1$. This is done by applying a smooth partition of unity on $[0,1]$.
We can easily arrange that $g^1_t$ be a nonconstant function for each $t$ 
and $(h^1_t,g^1_t)$ agrees with the initial pair $(h_t,g_t)$ for $t$ near $0,1$.

%
%
\smallskip\noindent
{\em Step 4.} 
We shall deform the path $(h^1_t,g^1_t)$ $(t\in [0,1])$ of meromorphic functions
from step 3, keeping it fixed to the second order at the points in $P^t_1\cap U$ for all $t$, 
and on $U$ for $t$ near $0$ and $1$, to a path of immersions  $(h_t,g_t):U\to(\CP^1)^2$ 
such that the 1-forms $h_t \, dg_t$ have vanishing periods on a system of 
curves forming a homology basis of $K_1$. 
(The neighbourhood $U$ is allowed to shrink around $K_1$.)
Then, $F_t=\Fscr(h_t,g_t)$ for $t\in [0,1]$ (see \eqref{eq:Fhg})
is a path of holomorphic Legendrian immersions from a neighbourhood of $K_1$ into $\CP^3$ 
which agrees with the given path for $t$ near $0$ and $1$. 

The deformation will consist of two substeps. In the first one we shall obtain a path of immersions
$U\to (\CP^1)^2$, and in the second one we will modify it (through a path of immersions) to one 
satisfying also the period vanishing conditions. 

For substep 1 we consider paths $(h_t,g_t)$ of the form
\begin{equation}\label{eq:deformation1}
	h_t=h^1_t+f_t \xi_t,\qquad g_t=g^1_t + f_t \eta_t, 
\end{equation}
where $f_t\in\Oscr(M)$ is a path of holomorphic functions vanishing to the second order at the
points of $P^1_t$ and nowhere else (such a path exists by Lemma \ref{lem:Weierstrass})
and $\xi_t,\eta_t\in \Oscr(U)$ are paths of holomorphic functions to be chosen.
Note that every such map is already an immersion into $(\CP^1)^2$ in small neighbourhoods
of the points in $P^t_1\cap U$ for all $t$. For a generic choice of the pair $\xi_t,\eta_t\in \Oscr(U)$ 
near the zero function, the map $(h_t,g_t):U\to (\CP^1)^2$ is then  an
immersion by H.\ Whitney's general position theorem  \cite{Whitney1936}. Indeed,
the domain of the map has real dimension $3$ (including the $t$ variable) and the maps are smooth, 
so the derivative with respect to the variable $x\in U$ of a generic map $(h_t,g_t)$ 
of this kind misses the origin $0\in \C^2$, the latter being of real codimension $4$. 

To simplify the notation, we denote the resulting path of immersions $U\to (\CP^1)^2$ again by $(h_t,g_t)$.
We may assume that $g_t$ is nonconstant for each $t$. In substep 2 we keep $g_t$ fixed 
and consider deformations of $h_t$ of the form
\begin{equation}\label{eq:deformation2}
	\tilde h_t=h_t + m_t \xi_{t}\quad\text{for}\  t\in [0,1], 
\end{equation}
where $m_t\in\Oscr(U)$ is a path of holomorphic functions vanishing to the second order 
at the points in $P^1_t\cap U$, and also at every critical point of $g_t$ in $U$ which is not a pole of $g_t$,
while $\xi_t\in \Oscr(U)$ is a path of holomorphic functions. A suitable path of multipliers $m_t$ is given by
\[
	m_t = (f_t)^3 (g'_t)^2 \quad\text{for}\ t\in [0,1],
\]
where $f_t\in\Oscr(M)$ is a path of holomorphic functions vanishing to the second order at the
points of $P^1_t$ (such a path exists by Lemma \ref{lem:Weierstrass}) and $g'_t=dg_t/d\zeta$. 
(Here, $\zeta:M\to\C$ is an immersion chosen at the beginning of the proof.)
Indeed, at any (simple) pole of $g_t$ the derivative $g'_t$ has a second order pole,
and since $f_t$ has a second order zero there, $m_t$ vanishes to order $6-4=2$.
On the other hand, at a critical point of $g_t$ which is not a pole, the function 
$(g'_t)^2$ has a second order zero, and hence so does $m_t$. We have that 
\[	
	d\tilde h_t = dh_t + \xi_t dm_t  + m_t d\xi_t.
\] 
At any critical point of $g_t$ not in $P^1_t$, the second and the third term on 
the right hand side vanish but $dh_t$ does not (since $(h_t,g_t)$ is an immersion at
such a point), and hence $d\tilde h_t$ does not vanish either. 
This shows that any choice of path $\xi_t$ in \eqref{eq:deformation2} furnishes a path of 
immersions $(\tilde h_t,g_t):U\to(\CP^1)^2$, and at the poles these functions have only changed 
for a second order term which does not affect the residues of $h_t dg_t$ (these remain zero).

It remains to choose the path $\xi_t\in \Oscr(U)$ in \eqref{eq:deformation2} such that the 
$1$-form $\tilde h_t dg_t$ has vanishing periods on a homology basis of $K_1$.
This can be done by the method in \cite[Sect.\ 4]{AlarconForstnericLopez2017CM}. 
Indeed, since we are deforming our maps only on the complement of the sets of poles, 
we are dealing with the standard contact form \eqref{eq:betast} on $\C^3$ and the results in 
\cite{AlarconForstnericLopez2017CM} apply. 
One uses the convex integration method along with period dominating sprays
and the parametric Mergelyan approximation theorem.
(A proof of the parametric Mergelyan approximation theorem for maps
to any complex manifold is spelled out in \cite[Theorem 4.3]{Forstneric2019IM}.) 
Note that the problem is linear in $\xi_t$, so we may use partitions of
unity in the $t$ variable. This reduces the problem to small subintervals of $[0,1]$
where it is almost the same as the problem for a single map 
treated in \cite{AlarconForstnericLopez2017CM} (since the poles
vary smoothly with $t$).  In particular, locally in $t$ we can choose 
the homology basis of $K_1$ in the complement of $P^1_t$. 
(The parametric case for Legendrian immersions is treated in 
more detail in \cite{ForstnericLarusson2018MZ}.)

This completes the initial stage of the induction. Let us denote the resulting path of meromorphic
functions, defined on a neighbourhood of $K_1$, again by $(h^1_t,g^1_t)$. 
The construction ensures that $(h^1_t,g^1_t)$ agrees with the initial family $(h_t,g_t)$ 
near the endpoints of $[0,1]$.

In the second stage of the induction, we construct a path $(h^2_t,g^2_t)$ $(t\in [0,1])$ of the
same type on a neighbourhood of $K_2$ which approximates $(h^1_t,g^1_t)$ from stage 1 
on $K_1$, agrees with it near $t=0,1$, and satisfies conditions (i)--(iii) 
(stated at the beginning of the proof) on the set $K_2$. This can be done by essentially the same 
procedure as in the initial stage,  but using also the parametric Runge approximation theorem,  
which is a special case of the parametric Oka-Weil theorem; see \cite[Theorem 2.8.4]{Forstneric2017E}). 
In Step 1, the sets $A^2_t,B^2_t,C^2_t$ 
must be defined so that they agree with $A^1_t,B^1_t,C^1_t$ in a neighbourhood of $K_1$
(this amounts to redefining the sets from the initial stage in the complement of $K_1$).
Let $P^2_t = A^2_t\cup B^2_t\cup C^2_t$; so $P^2_t\cap K_1= P^1_t\cap K_1$ for all $t$.
In step 2, we extend the jets of $(h^1_t,g^1_t)$ $(t\in [0,1])$ smoothly from $P^1_t\cap K_1$ to 
$P^2_t\cap K_2$ so that conditions \eqref{eq:Res} hold.
When solving the $\dibar$-problem in step 3, we correct the solutions by using the parametric
Runge theorem so that they approximate those from the first stage on $K_1$. Step 4 can be 
handled by the same tools, using period dominating sprays and the parametric 
Mergelyan approximation theorem in order to preserve the period vanishing conditions on $K_1$
and in addition fulfil those on the new curves in the period basis for $K_2$. The details
are similar to those in \cite{ForstnericLarusson2018MZ} and we omit them.

Proceeding in the same way, we obtain a sequence of solutions $(h^m_t,g^m_t)$ $(t\in [0,1])$ on 
$K_m$ $(m\in\N)$ which approximates $(h^{m-1}_t,g^{m-1}_t)$ on $K_ {m-1}$ and agrees
with the initial data $(h_t,g_t)$ for $t$ near $0$ and $1$. Assuming as we may
that  the approximations are close enough at every step, the sequence $(h^m_t,g^m_t)$ converges 
to a solution $(h_t,g_t)$ on $M$ as $m\to\infty$.
\end{proof}

%
%
\section{The homotopy principle}\label{sec:HP}
Let $M$ be an open Riemann surface.  Let $\Lscr_\textrm{formal}(M,\CP^3)$ be the space of formal holomorphic Legendrian immersions from $M$ to $\CP^3$, that is, commuting squares
\[ 
	\xymatrix{
	TM \ar[r]^\phi \ar[d] & \xi \ar[d] \\ M \ar[r]^f & \CP^3
	} 
\]
where $\xi$ is the contact subbundle of $T\CP^3$, $\phi$ is a monomorphism, and $f$ is holomorphic.  
In this section we show that the inclusion
\[
	 \Lscr(M,\CP^3) \hookrightarrow \Lscr_\textrm{formal}(M,\CP^3) 
\]
induces a bijection of path components.

There are very few results of this kind in the literature.  The full parametric h-principle holds for Legendrian 
holomorphic immersions of $M$ into $\C^{2n+1}$ with the standard complex contact structure 
(see \cite{ForstnericLarusson2018MZ}).  Here, crucial use is made of the projection $\C^{2n+1}\to\C^{2n}$, 
$(x,y,z)\mapsto (x,y)$ (the standard contact form is $dz+xdy$).  For plain maps, not necessarily immersions, 
the h-principle is obvious.  

There are also results for projectivised cotangent bundles with the standard complex contact structure 
(see \cite{ForstnericLarusson2018X}).  The inclusion of the space of holomorphic Legendrian maps $M\to \P T^*Z$, 
where $Z$ is a manifold with $\dim Z\geq 2$, into the space of formal holomorphic Legendrian maps induces 
a surjection of path components.  For closed holomorphic curves that are strong immersions, the inclusion 
induces a bijection of path components.  And if $\dim Z\geq 3$, the inclusion also induces an epimorphism 
of fundamental groups, but this fails in general when $\dim Z=2$.  Here, crucial use is made of the projection 
$\P T^*Z\to Z$ and the fact that its fibres are Oka.

Consider now formal holomorphic immersions of $M$ into an Oka manifold $Y$, directed by a subbundle $\xi$ 
of $TY$.  Trivialise $TM$ once and for all.  Then a formal holomorphic immersion $M\to Y$, directed by $\xi$, 
is nothing but a holomorphic map $M\to E$, where $E$ is the holomorphic fibre bundle over $Y$ obtained 
from $\xi$ by removing the zero section.  The fibre of $E$ is $\C^k\setminus\{0\}$, where $k$ is the rank of $\xi$.
Hence $E$ is an Oka manifold, so the inclusion $\Oscr(M,E) \hookrightarrow \mathscr C(M,E)$ is a weak equivalence.  
Determining the weak homotopy type of the space of formal holomorphic immersions $M\to Y$, directed by $\xi$, 
is thus reduced to a purely topological problem.

The long exact sequence of homotopy groups
\[ \cdots \to \pi_1(\C^k\setminus\{0\}) \to \pi_1(E) \to \pi_1(Y) \to \cdots  \]
shows that if $Y$ is simply connected and $k\geq 2$, then $E$ is simply connected, so $\Oscr(M,E)$ is path 
connected.  The basic h-principle follows, as long as there is at least one genuine holomorphic immersion 
$M\to Y$, directed by $\xi$.

Now we return to holomorphic Legendrian immersions of $M$ into $\CP^3$.  Here, of course, $\CP^3$ is simply 
connected and $k=2$, so $\Lscr_\textrm{formal}(M,\CP^3)$ is path connected.  Also, by Theorem \ref{th:connected}, 
$\Lscr(M,\CP^3)$ is path connected and clearly nonempty (consider for example $\Bscr(1,g)=[1,1,g,0]$, where 
$g:M\to\C$ is a holomorphic immersion, as provided by the theorem of 
Gunning and Narasimhan \cite{GunningNarasimhan1967}).  Thus we have the following h-principle.

\begin{theorem}\label{th:h-principle}
Every formal holomorphic Legendrian immersion from an open Riemann surface to $\CP^3$ can be 
deformed to a genuine holomorphic Legendrian immersion, unique up to homotopy.
\end{theorem}

%
%
\section{Calabi--Yau property and complete immersions}\label{sec:CY}

In this and the following sections we discuss implications of our new results, as well as 
those from some other recent papers, to the existence of complete Legendrian 
curves in $\CP^3$ and conformally immersed superminimal surfaces in the $4$-sphere $\S^4$.

We begin by discussing completeness of immersions on a formal level, with the aim 
of conceptualising and unifying phenomena of this type in different geometries. 

Let $(N,g)$ be a connected smooth Riemannian manifold of dimension $n$, possibly endowed with 
some additional structure. For example, $N$ could be a complex manifold, a complex contact manifold, 
a manifold with a chosen subset of the tangent bundle, etc. 

Assume that $M$ is a smooth manifold of dimension $\dim M<n=\dim N$.
For every immersion $F:M\to (N,g)$ we have the induced Riemannian metric $F^*g$ and distance function
$\dist_F$ on $M$. Assume now that $M$ is either noncompact, or a compact manifold with nonempty 
smooth boundary. For a fixed interior point $p_0\in M$ we denote by 
\begin{equation}\label{eq:radius}
	R_F(M,p_0)\in (0,+\infty]
\end{equation}
the {\em intrinsic radius} of $M$, defined as the infimum of lengths (in the 
metric $F^*g$) of all divergent paths $\gamma:[0,1)\to M$ with $\gamma(0)=p_0$. 
(A path $\gamma$ is said to be {\em divergent} if the point $\gamma(t)\in M$ leaves any compact subset of 
the interior of $M$ as $t\to 1$.) If $M$ is a compact manifold with boundary $bM$ and 
$F:M\to N$ is  an immersion, then $R_F(M,p_0) = \dist_F(p_0,bM)$ is simply the distance from $p_0$
to $bM$ in the metric $F^*g$. Changing the base point clearly changes the intrinsic radius by an additive 
constant which is irrelevant in our considerations. 

An immersion $F:M\to N$ is said to be {\em complete} if $F^*g$ is a 
complete metric on $M$, i.e., the induced distance function $\dist_F$ makes $M$ into a complete 
metric space. If $M$ is an open manifold (noncompact and without boundary), this is equivalent to 
$R_F(M,p_0)=+\infty$. We refer to M.\ do Carmo \cite{doCarmo1992} where these concepts are explained 
in more detail.

Consider a class $\Fscr(\cdotp,N)$ of $\Cscr^1$ immersions from smooth manifolds $M$ 
of dimension $\dim M<n=\dim N$, possibly with boundary, into a given manifold $N$.  
(Typically one considers immersions which are solutions of some elliptic PDE, 
so they are smooth in the interior of $M$.)
For a given $M$, we denote by $\Fscr(M,N)$ the space of immersions $M\to N$ of this class, 
endowed with the $\Cscr^1$ topology.  The source manifolds may also carry additional geometric structures. 
For example, they may be conformal surfaces or Riemann surfaces, the case of main interest to us.  
As for the class $\Fscr(\cdotp,N)$, we could be considering for example conformal minimal
immersions from conformal surfaces to a Riemannian manifold $(N,g)$, 
or holomorphic immersions from Riemann surfaces into a complex manifold $N$, 
or holomorphic Legendrian immersions into a complex contact manifold $(N,\xi)$,
or holomorphic null curves $M\to N=\C^n$ with $n\ge 3$, etc. We shall say that $M$ and $N$ are {\em admissible}
for the given class of immersions if the definition of the class makes sense for them.
For example, when considering holomorphic immersions, our manifolds must be complex, 
and for conformal immersions, they must be conformal manifolds. The precise smoothness class 
of manifolds and immersions may depend on the situation. 

We shall assume the following conditions on a class $\Fscr(\cdotp,N)$ that we wish to consider.

%
%
\begin{enumerate}[\rm (a)]
\item 
If $F\in \Fscr(M,N)$ and $M_0\subset M$ is either an open domain or a compact smoothly bounded domain, 
then $F|_{M_0}\in \Fscr(M_0,N)$. Conversely, if $F:M\to N$ is an immersion which is of class 
$\Fscr(\cdotp,N)$ on an open neighbourhood of any point of $M$, then $F\in \Fscr(M,N)$. 
\item 
If $M$ is a compact admissible manifold with boundary, then $\Fscr(M,N)$ is nonempty.
\item
If a sequence $F_j\in \Fscr(M,N)$ $(j\in\N)$ converges in the $\Cscr^1(M,N)$ topology on compacts in $M$ 
to an immersion $F:M\to N$, then $F\in \Fscr(M,N)$.
\item
(Interior estimates.) 
Let $g_0$ be a Riemannian metric on $M$. Given $F\in \Fscr(M,N)$, a pair of relatively 
compact domains $M_0\Subset M_1\subset M$ and a number $\epsilon>0$, there is $\delta>0$
such that for any $G\in \Fscr(M_1,N)$, we have that
\begin{equation}\label{eq:elliptic}
	\max_{p\in M_1}\dist_g(F(p),G(p)) < \delta 
	\ \Longrightarrow\   \max_{p\in M_0} \dist_{g_0,g}(dF_p,dG_p) <\epsilon.
\end{equation}
\end{enumerate}

Condition (a) says that immersions of class $\Fscr$ are sections of a sheaf of immersions.
Condition (b) is typically fulfilled by the existence of $\Fscr$-immersions $M\to N$ with values in a chart of $N$. 
Condition (c) says that  $\Fscr(M,N)$ is closed in the space of all immersions $M\to N$
in the $\Cscr^1$ topology. Condition (d) means that the distance between
$F$ and $G$ in the $\Cscr^1$ topology on the smaller domain $M_0$ can be estimated by the 
distance between the two maps in the $\Cscr^0$ topology (i.e., the uniform distance) on the bigger 
domain $M_1$. This is a typical elliptic type estimate which holds whenever our immersions are solutions 
of some elliptic PDE; in particular, it holds for harmonic and holomorphic maps.
 
We have already mentioned the Calabi--Yau problem for minimal surfaces in the introductory 
section. We now introduce the following key condition which lies behind all recently established 
Calabi--Yau-type theorems in various geometries.

%
%
\begin{definition}[Calabi--Yau property] \label{def:CY}
Assume that $(N,g)$ is a Riemannian manifold and $\Fscr(\cdotp,N)$ is a class of immersions 
into $N$ satisfying conditions (a)--(d) above. The class $\Fscr(\cdotp,N)$ enjoys 
the {\em Calabi--Yau property} if the following holds true. Given a compact $\Fscr$-admissible 
manifold $M$ with boundary $bM$, a point $p_0\in M\setminus bM$,  
an immersion $F_0 \in \Fscr(M,N)$, and numbers $\epsilon >0$ (small) 
and $\lambda>0$ (big), there exists an immersion $F\in \Fscr(M,N)$ such that 
\begin{equation}\label{eq:increase}
	\dist_g(F,F_0):=\max_{p\in M}\dist_g(F(p),F_0(p))<\epsilon \ \ \text{and}\ \ 
	R_F(M,p_0) > \lambda.
\end{equation}
\end{definition}
 
Recall that $R_F$ denotes the intrinsic radius \eqref{eq:radius} of the immersion $F$.

The following result may be viewed as an {\em abstract Calabi--Yau theorem}.
It is motivated by the classical Calabi--Yau problem for minimal surfaces, and it
summarises all recent results on this subject in the literature (see Example \ref{ex:CY}).
For the history of this problem, see the discussion and references in 
\cite{AlarconForstneric2019JAMS,AlarconForstneric2019RMI,AlarconForstnericLopez2021}.

%
%
\begin{theorem}\label{th:abstractCY}
Assume that $(N,g)$ is a Riemannian manifold and $\Fscr(\cdotp,N)$ is a class of immersions 
into $N$ satisfying conditions (a)--(d) above and the Calabi--Yau property
(see Definition \ref{def:CY}). Let $M$ be a compact $\Fscr$-admissible manifold with boundary.
Then, every $F_0\in \Fscr(M,N)$ can be approximated as closely as desired uniformly on $M$ by a
continuous map $F:M\to N$ such that $F|_{M^\circ}: M^\circ=M\setminus bM \to N$ 
is a complete immersion in $\Fscr(M^\circ,N)$. 

If in addition the immersion $F$ in \eqref{eq:increase} can always be 
chosen injective on $M$ or $bM$, then $F$ as above can be chosen injective on $M$ or $bM$, respectively.

If in addition the immersion $F$ in \eqref{eq:increase} can always be chosen to agree with $F_0$ to a given finite order 
at each point in a given finite subset of $M^\circ$, then $F$ as above can also be so chosen.

Furthermore, if $M_0$ is a domain in $M^\circ$ obtained by removing 
from $M^\circ$ a countable family of pairwise disjoint, compact, smoothly bounded domains $D_j$, $j\in\N$,
then for every $F_0\in \Fscr(M,N)$ and $\epsilon>0$, there exists a continuous map
$F:\overline M_0\to N$ such that $\dist_g(F,F_0|_{\overline M_0})<\epsilon$ and 
$F|_{M_0}:M_0\to N$ is a complete immersion in $\Fscr(M_0,N)$.
\end{theorem}

\begin{proof}
The first statement is seen by following  
\cite[proof of Theorem 1.1]{AlarconDrinovecForstnericLopez2015PLMS}.
Indeed, the Calabi--Yau property allows one to construct a sequence
of immersions $F_j\in \Fscr(M,N)$ $(j\in\N)$ which converges uniformly
on $M$ to a continuous map $F:M\to N$ such that 
\begin{equation}\label{eq:radiigrow}
	\lim_{j\to\infty} R_{F_j}(M,p_0)=+\infty.
\end{equation}
Assuming as we may that the approximation of $F_j$ by $F_{j+1}$ is sufficiently close in every step, 
condition (d) on the class $\Fscr(\cdotp,N)$ (see in particular \eqref{eq:elliptic}) implies that
the restrictions of $F_j$ to any relatively compact subset of $M^\circ$ converge
in the $\Cscr^1$ topology to an immersion, and hence $F|_{M^\circ}\in \Fscr(M^\circ,N)$ 
in view of condition (a). Completeness of the limit immersion $F|_{M^\circ}:M^\circ\to N$ 
follows from \eqref{eq:radiigrow} in view of \cite[Lemma 2.2]{AlarconForstneric2019RMI} 
which shows that the intrinsic radius $R_{F_j}(M,p_0)$ 
can decrease only a little under $\Cscr^0$-small deformations of the map. 
(This is obvious for deformations which are small in the $\Cscr^1$ norm, but the point is that it also 
holds for $\Cscr^0$-small deformations.) Alternatively, one can apply the argument in 
\cite[proof of Theorem 1.1]{AlarconDrinovecForstnericLopez2015PLMS},
which controls from below the intrinsic radii of an increasing sequence of compact 
domains in $M$ exhausting $M^\circ$, using the fact that the convergence
$F_j\to F$ is in the $\Cscr^1$ topology on each compact subset of $M^\circ$. 
There, it is also explained how to obtain injectivity of the limit map $F$ on $M$ or $bM$, provided 
the immersions $F_j$ in the sequence can be chosen injective and the uniform approximation is close
enough at each step.

The second statement is seen by \cite[proof of Theorem 1.1]{AlarconForstneric2019RMI}
where this was proved for conformal minimal immersions from Riemann surfaces to flat Euclidean spaces $\R^n$. 
Fix a point $p_0\in M_0$ and for $k\in\N$ consider the decreasing sequence of 
domains $M_k=M^\circ \setminus \bigcup_{j=1}^k D_j$. By using the Calabi--Yau property
(see in particular \eqref{eq:increase}) we construct a sequence of immersions $F_k\in \Fscr(\overline M_k,N)$, $k=1,2,\ldots$, converging uniformly on $\overline M_0 = \bigcap_{k}\overline M_k$ to a continuous map 
$F:\overline M_0\to N$ and such that 
\begin{equation}\label{eq:Mk-toinfinity}
	\lim_{k\to\infty}R_{F_k}(M_k,p_0)=+\infty.
\end{equation}
Assuming as we may that the convergence $F_k\to F$ on $\overline M_0$ is fast enough, the interior 
estimates \eqref{eq:elliptic} ensure that the sequence $F_k$ converges in the $\Cscr^1$ topology on any 
compact subset of $M_0$ to an immersion, and hence $F|_{M_0}\in \Fscr(M_0,N)$ by condition (a). 
Finally, from \eqref{eq:Mk-toinfinity} and \cite[Lemma 2.2]{AlarconForstneric2019RMI} 
it follows that $F|_{M_0}$ is a complete immersion.
\end{proof}

\begin{remark}
Since the immersions $F$ in Theorem \ref{th:abstractCY} have ranges contained in a
compact neighbourhood of the range $F_0(M)$ of the initial immersion, and since any two metrics on $N$ are
comparable on a compact set, the approximating immersions in Theorem \ref{th:abstractCY} 
are complete in any given Riemannian metric on $N$. 
\end{remark}

%
%
\begin{example} \label{ex:CY}
The following classes of manifolds and immersions are known to enjoy the Calabi--Yau 
property, and hence the conclusion of Theorem \ref{th:abstractCY} holds for them. 
\begin{enumerate}[(\rm i)]
\item $N=\R^n$ with $n\ge 3$, $M$ is a compact conformal surface with boundary
(or a compact bordered Riemann surface), and $\Fscr(M,\R^n)$ is the space
of conformal minimal immersions $M\to \R^n$.
See \cite[Theorem 1.1]{AlarconDrinovecForstnericLopez2015PLMS} for the orientable
case and \cite[Theorem 6.6]{AlarconForstnericLopezMAMS} for the nonorientable one.
Injectivity on $bM$ can be obtained for any $n\ge 3$, and injectivity on $M$ 
for any $n\ge 5$. 
\item $N=\C^n$, $M$ is a compact bordered Riemann surface, and $\Fscr(M,\C^n)$
is the space of holomorphic (or null holomorphic if $n\ge 3$) immersions;  
see \cite{AlarconForstneric2014IM}. In this case, injectivity on $M$ can be obtained for 
any $n\ge3$, and injectivity on $bM$ for any $n\ge 2$.
\item $N=\C^{2n+1}$ with  the standard complex contact structure  given by
\eqref{eq:alphast}, $M$ is a compact bordered Riemann surface,
and $\Fscr(M,\C^{2n+1})$ is the space of Legendrian immersions of class $\Cscr^1(M,\C^{2n+1})$ which 
are holomorphic on $M^\circ=M\setminus bM$ 
(see \cite[Theorem 1.2 and Lemma 6.5]{AlarconForstnericLopez2017CM}). 
The limit map can be chosen injective on $M$.
\item $(N,\xi)$ is an arbitrary complex contact manifold, $M$ is a compact smoothly bounded domain
in an open Riemann surface $\wt M$, and $\Fscr(M,N)$ is the space of holomorphic Legendrian immersions 
$F:U_F\to N$ on small open neighborhoods $U_F\subset \wt M$ of $M$.
As in the previous case, the limit map can be chosen injective on $M$.
Indeed, by \cite[Theorem 1.3]{AlarconForstneric2019IMRN}, the Calabi--Yau property is obtained from 
the standard case $N=\C^{2n+1}$ (case (iii)) by using a holomorphic Darboux neighbourhood 
of the immersed holomorphic Legendrian curve $\wt F$ (see \cite[Theorem 1.1]{AlarconForstneric2019IMRN}). 
\end{enumerate}
\end{example}

In all these examples, the Calabi--Yau condition was established by finding approximate solutions 
to the Riemann-Hilbert boundary value problem in the respective geometry, combined with the method
of exposing boundary points of compact bordered Riemann surfaces. 
The former is the most difficult part of the work, intricately depending on the geometric properties 
of manifolds and immersions. The main advantage of the Riemann-Hilbert modification method 
over other possible methods is that it provides very precise geometric control on the placement 
of the image of $M$ inside $N$, something
which was impossible by the earlier methods used in the Calabi--Yau problem for
minimal surfaces in Euclidean spaces. Most importantly, this technique allows one
to keep the source manifold $M$ and its associated structures (such as the conformal structure) 
unchanged. Sufficient conditions for the existence of injective immersions are 
obtained by proving a suitable general position theorem for a given class of immersions, 
and this typically depends on the dimensions of manifolds and other geometric conditions
associated to the given class of immersions.

The following is one of the most challenging questions in this subject.

%
%
\begin{problem}\label{prob:CY}
Let $\Fscr(\cdotp,N)$ be the class of conformal minimal immersions
from smooth, compact, bordered conformal surfaces into a smooth Riemannian manifold $(N,g)$ 
of dimension at least 3.  Does this class enjoy the Calabi--Yau property for every $(N,g)$?
\end{problem}

Although we do not see any a priori reasons against this being true, it seems that an 
(affirmative) answer is known only when $N$ is a flat Euclidean space (see Example \ref{ex:CY} (i)).
We will see in the following section that the Calabi--Yau property also holds for superminimal surfaces in 
the $4$-sphere with the spherical metric (see Theorem \ref{th:CYS4}).
After the completion of this paper, Forstneri\v c \cite{Forstneric2020JGEA} established the Calabi--Yau property
for conformal superminimal surfaces of appropriate spin in any self-dual or anti-self-dual Einstein $4$-manifold 
by using the techniques developed in this paper in the special case of the $4$-sphere. 
The key point is to use Corollary \ref{cor:CYLegendrian} together with an analogue of the 
Bryant correspondence, given by Theorem \ref{th:Bryant}, for this class of oriented Riemannian 4-manifolds.

A complex-analytic analogue of the Calabi--Yau problem is called {\em Yang's problem}, 
named after P.\ Yang \cite{Yang1977JDG} who in 1977 asked about the existence 
of complete bounded complex submanifolds in complex Euclidean spaces. 
There has been a surge of recent activity on this problem, initiated by A.\ Alarc{\'o}n and 
F.\ Forstneri\v{c} \cite{AlarconForstneric2013MA} in 2013, A.\ Alarc{\'o}n and 
F.\ J.\ L{\'o}pez \cite{AlarconLopez2016JEMS} in 2016 and, with a completely different method,
by J.\ Globevnik \cite{Globevnik2015AM} in 2015; see the survey in
\cite[pp.\ 291--292]{AlarconForstneric2019JAMS}. In some of these works --- 
see especially \cite{AlarconJDG,AlarconForstneric2019MZ,AlarconGlobevnik2017,AlarconGlobevnikLopez2019Crelle} ---  
a weaker analogue of the Calabi--Yau property was established by a different technique, using holomorphic 
automorphisms of complex Euclidean spaces to successively deform a given complex submanifold so that
it avoids more and more pieces of a certain labyrinth, thereby increasing its intrinsic radius
and making it complete in the limit. The advantage of this method, when compared to the 
Riemann-Hilbert method, is that it  preserves embeddedness, but the disadvantage is that one 
must cut away pieces of the source manifold to keep the image bounded, so one loses control of its 
complex structure.

Immersions of types (i) and (ii) in Example \ref{ex:CY} are known to satisfy the 
interpolation condition in Theorem \ref{th:abstractCY}. We now show that the 
classes (iii) and (iv) also satisfy this condition. The following is an extension of 
\cite[Lemma 4.1]{AlarconForstneric2019IMRN}.

%
%
\begin{lemma}[Calabi--Yau property with interpolation for holomorphic Legendrian immersions]
\label{lem:completeness}
Let $N$ be a complex contact manifold endowed with a Riemannian metric. Also, let $M$ be a compact 
bordered Riemann surface, $E\subset   M^\circ=M\setminus bM$ be a finite set, $p_0\in M^\circ$ be a point, 
and $F_0: M\to N$ be a holomorphic Legendrian immersion on an neighborhood of $M$ in an 
ambient Riemann surface. Given a number $\lambda>0$ (big), $F_0$ can be approximated uniformly 
on $M$ by holomorphic Legendrian immersions $F:  M\to N$ satisfying the following conditions.
\begin{enumerate}[\rm (i)]
\item $\dist_F(p_0,bM)>\lambda$. 
\item $F$ agrees with $F_0$ to  any given finite order at every point of $E$.
\end{enumerate}
Furthermore, if $F_0|_E:  E\to N$ is injective then $F:M\to N$ can be chosen an embedding.
\end{lemma}

%
%
\begin{proof}
The novelty with respect to \cite[Lemma 4.1]{AlarconForstneric2019IMRN} is condition {\rm (ii)}. 
When $N=\C^{2n+1}$ with the Euclidean metric, the lemma coincides
(except for condition {\rm (ii)}) with \cite[Lemma 6.5]{AlarconForstnericLopez2017CM} which
holds true for any compact bordered Riemann surface (see the discussion at the beginning of 
\cite[Sec.\ 6]{AlarconForstnericLopez2017CM}). The interpolation condition (ii)
is easily achieved by the techniques developed in \cite{AlarconForstnericLopez2017CM}.
It is the same technique which gives holomorphic immersions $(x,y):M\to \C^{2n}$
for which $xdy=\sum_{j=1}^n x_j dy_j$ is an exact $1$-form on $M$;
any such defines a Legendrian immersion $F=(x,y,z):M \to\C^{2n+1}$ with the last component
$z=-\int xdy$. To achieve the interpolation conditions, 
we arrange in addition that the immersion $(x,y)$ has correct jets at points of the given finite
set $E$ (matching those of the initially given immersion to specified orders), 
and the integral of $xdy$ has suitably prescribed values on a collection
of arcs in $M$ connecting a base point $p_0\in M$ to the points in $E$. 
The last condition, which is achieved by the methods in 
\cite[proof of Theorem 5.1]{AlarconForstnericLopez2017CM}, can be used
to ensure that the last component function $z=-\int xdy$ also has correct
values at the points of $E$; the jet intepolation condition for $z$ at the points of $E$
then follows immediately from those for $(x,y)$. For the details in a similar setting, see 
\cite{AlarconCastroInfantes2019APDE}.

This proves the lemma for $N=\C^{2n+1}$. 
It is shown in \cite[Theorem 1.1]{AlarconForstneric2019IMRN} that every complex contact manifold $N$ 
admits a holomorphic Darboux chart around any immersed noncompact holomorphic Legendrian curve. 
Using such charts, the general case of the lemma is obtained by following word for word the proof 
of \cite[Lemma 4.1]{AlarconForstneric2019IMRN}, but applying the special case of 
Lemma \ref{lem:completeness} for $N=\C^{2n+1}$ instead of \cite[Lemma 6.5]{AlarconForstnericLopez2017CM}.
\end{proof}

In view of Theorem \ref{th:abstractCY}, Lemma \ref{lem:completeness} implies the following Calabi--Yau 
type theorem for holomorphic Legendrian curves in any complex contact manifold. Except for the 
interpolation condition, the statement for finite bordered Riemann surfaces
coincides with \cite[Theorem 1.3]{AlarconForstneric2019IMRN}, while the part
for surfaces with countably many boundary curves is new.

%
%
\begin{corollary}[Calabi--Yau theorem for Legendrian curves]\label{cor:CYLegendrian}
Holomorphic Legendrian immersions from any compact bordered Riemann surface
into an arbitrary complex contact manifold satisfy the conclusion of 
Theorem \ref{th:abstractCY}.

In particular, if $M$ is an open Riemann surface of finite genus with at most countably
many ends, none of which are point ends, then $M$ admits a complete injective holomorphic 
Legendrian immersion into any complex contact manifold.
\end{corollary}

By the uniformisation theorem of X.\ He and O.\ Schramm \cite{HeSchramm1993},
every open Riemann surface as in the second part of the above corollary is conformally equivalent 
to a domain in a compact Riemann surface with at most countably many closed geometric discs removed.
Hence, it is of the kind as in the last statement in Theorem \ref{th:abstractCY}, so that result applies. 

We now introduce the notion of a Runge exhaustion for a given class of immersions. 

%
%
%
%
\begin{definition}\label{def:Runge}
Let $M$ be an admissible manifold for  a class of immersions $\Fscr(\cdotp,N)$.
An exhaustion $M_1\subset M_2\subset \cdots \subset \bigcup_{j=1}^\infty M_j =M$
of $M$ by compact smoothly bounded domains is an {\em $\Fscr(\cdotp,N)$-Runge exhaustion}
if for every $j\in \N$ we have $M_j\subset  M_{j+1}^\circ$ and every $F\in \Fscr(M_j,N)$ 
can be approximated in $\Cscr^1(M_j,N)$ by immersions $G\in \Fscr(M_{j+1},N)$.
The exhaustion satisfies the {\em interpolation condition} if, in addition, for every $F$ as above
the immersion $G$ can be chosen to agree with $F$ to a given finite order at a given finite set 
of points in $M_j^\circ$.
\end{definition}

Note that the definition tacitly includes the topological condition concerning extendibility of maps
from sets in the given exhaustion.

For holomorphic or (conformal) harmonic immersions from open Riemann surfaces, one typically
tries to show that any exhaustion by compact sets without holes in $M$ is a Runge exhaustion. 
This holds for instance for holomorphic immersions $M\to\C^n$, 
null holomorphic immersions $M\to\C^n$ for $n\ge 3$ \cite[Corollary 2.7]{AlarconForstneric2014IM}, 
conformal minimal immersions into $\R^n$ for any $n\ge 3$ (see \cite{AlarconLopez2012JDG} for $n=3$ 
and \cite[Theorem 5.3]{AlarconForstnericLopez2016MZ} for the general case), and 
holomorphic Legendrian immersions into complex Euclidean or complex projective spaces 
with their standard complex contact structures (see \cite{AlarconForstnericLopez2017CM} 
for $\C^{2n+1}$ and Section \ref{sec:approx-interpol} of this paper for $\CP^3$). 

We have the following Runge approximation theorem for (complete) $\Fscr$-immersions.

%
%
\begin{theorem}\label{th:complete2}
Assume that $(N,g)$ is a Riemannian manifold and $\Fscr(\cdotp,N)$ is a class of immersions into $N$.
If $M$ is an open $\Fscr$-admissible manifold which admits an $\Fscr(\cdotp,N)$-Runge exhaustion 
$(M_j)_{j\in \N}$ (see Definition \ref{def:Runge}), then every immersion $F_i\in\Fscr(M_i,N)$ with $i\in \N$ 
can be approximated in $\Cscr^1(M_i,N)$ by immersions $F\in \Fscr(M,N)$.
If in addition the class $\Fscr(\cdotp,N)$ enjoys the Calabi--Yau property (see Definition \ref{def:CY}),
then $F$ can be chosen complete. 
\end{theorem}

\begin{proof}
Let $(M_j)_{j\in \N}$ be an $\Fscr$-Runge exhaustion of $M$ and $F_i\in\Fscr(M_i,N)$ for some $i\in\N$. 
Assume that the class $\Fscr(\cdotp,N)$ enjoys the Calabi--Yau property. 
By alternately using the Runge exhaustion property (see Definition \ref{def:Runge}) and the Calabi--Yau property
(Definition \ref{def:CY}), we find a sequence $F_j\in\Fscr(M_j,N)$ with $j\ge i$ such that for every 
$j=i,i+1,\ldots$, the restriction $F_{j+1}|_{M_j}$ approximates $F_j$ as closely as desired in $\Cscr^1(M_j,N)$ and the 
intrinsic diameter of $M_{j+1}$ with respect to $F_{j+1}$ is arbitrarily large. By doing this in the right way, 
the sequence $F_j$ converges in $\Cscr^1(M_k,N)$ for every $k\ge i$ to a complete immersion $F\in \Fscr(M,N)$ 
which approximates $F_i$ on $M_i$. If the Calabi--Yau property hold for embeddings, 
we can obtain a complete embedding $M\hra N$ in $\Fscr(M,N)$. In the absence of the Calabi--Yau property,
the above argument still holds without completeness and yields $F\in \Fscr(M,N)$
approximating $F_i$ on $M_i$.
\end{proof}

%
%
\begin{remark}\label{rem:metric}
(a) Since the Calabi--Yau property pertains to maps with 
ranges in a relatively compact neighbourhood of the range of a given map,
we see that the immersion $F\in \Fscr(M,N)$ in Theorem \ref{th:complete2}
can be chosen complete in any given metric on $N$. 

\noindent (b) In many cases of interest, it is possible to include the jet interpolation condition into 
the Runge approximation property and thereby obtain a version of Theorem \ref{th:complete2} 
with jet interpolation on infinite closed discrete subsets of $M$. 
This typically requires one to refine the exhaustion at each inductive step by adding new intermediate sets. 
\qed\end{remark}

%
%
\begin{corollary}[Complete embedded Legendrian curves]\label{cor:completeLeg}
\begin{enumerate}[\rm (i)]
\item
Every Riemann surface admits a complete injective holomorphic Legendrian 
immersion into $\CP^3$. 
\item
In Theorem \ref{th:Runge-for-immersions} (the Runge approximation theorem 
for Legendrian immersions of open Riemann surfaces into $\CP^3$), the approximating immersion 
can be chosen complete.
\item
Every open Riemann surface admits a complete injective holomorphic Legendrian 
immersion into $\CP^3$ with (everywhere) dense image. 
\item
Let $N$ be a connected complex contact manifold and $M$ be an open Riemann surface. 
Assume that every regular exhaustion of $M$ by compact domains without holes is a Runge exhaustion
with interpolation for holomorphic Legendrian immersions into $N$ (see Definition \ref{def:Runge}).
Then, there exists a complete injective holomorphic Legendrian immersion $M\to N$ with dense image.
\end{enumerate}
\end{corollary}

\begin{proof}
For a compact Riemann surface, (i) holds by Bryant's theorem \cite[Theorem G]{Bryant1982JDG}. 
For an open Riemann surface, it is seen by combining Theorem \ref{th:Runge-for-immersions} 
(the Runge approximation theorem for holomorphic Legendrian immersions into $\CP^3$), 
Lemma \ref{lem:completeness}  (the Calabi--Yau property with interpolation for Legendrian immersions),   
the general position theorem for holomorphic Legendrian immersions into an arbitrary
complex contact manifold (see \cite[Theorem 1.2]{AlarconForstneric2019IMRN}),
and the proof of Theorem \ref{th:complete2}. The same argument gives (ii).

To obtain (iii), we apply the same proof but add also the interpolation condition at finitely many points 
at every step of the inductive construction, adding more and more points of a given closed discrete subset $E$
in the open Riemann surface $M$ as we proceed. This also requires one to refine the exhaustion at 
every step of the induction. In this way, we can arrange that the resulting injective
Legendrian immersion $F:M\to\CP^3$ interpolates a prescribed injective map $E\to\CP^3$ 
with dense image, and hence the curve $F(M)$ is dense in $\CP^3$. 
See \cite[Sect.\ 4.4]{AlarconCastroInfantes2018GT} for the details.
The same arguments apply in any complex contact manifold enjoying the Runge property 
for holomorphic Legendrian immersions from open Riemann surfaces, thereby giving part (iv).
\end{proof}

%
%

\section{Superminimal surfaces in the four-dimensional sphere}\label{sec:S4}

In this section, we apply some results from Sect.\ \ref{sec:CY} to establish the Calabi--Yau property and a 
Runge approximation theorem for conformal superminimal surfaces in the $4$-sphere $\S^4$. 
The proofs are based on the Bryant correspondence given by Theorem \ref{th:Bryant}.

We begin by recalling the construction and basic properties 
of the Penrose twistor map $\pi:\CP^3\to\S^4$; see e.g.\ \cite{Penrose2018, PenroseMacCallum1973}. 
We shall follow Bryant's paper \cite[Sect.\ 1]{Bryant1982JDG}, but the reader may also wish to consult
J.\ Bolton and L.\ M.\ Woodward \cite{BoltonWoodward2000} and J.\ C.\ Wood \cite{Wood1987}.
A self-contained exposition can also be found in \cite[Sect.\ 6]{Forstneric2020JGEA}
where additional references are provided.

Let $\H$ denote the algebra of quaternions. An element of $\H$ can be written uniquely as
\[
	q= x+\igot y +\jgot u +\kgot v = z+\jgot w,\quad \text{where}\ z=x+\igot y\in\C\ \text{and}\ w=u-\igot v\in \C.
\] 
Here, $\igot,\jgot,\kgot$ are the quaternionic units.
In this way, we identify $\H$ with $\C^2$ and the quaternionic plane $\H^2=\H\times \H$ with $\C^4$.
Write $\H^2_*=\H^2\setminus \{0\}$. Consider the following diagram:
\[ 
	\xymatrix{
	\C^* \,\, \ar@{^{(}->}[r] & \H^2_*=\C^4_* \ar[d]^{\phi}  \\ 
	\CP^1 \,\, \ar@{^{(}->}[r] & \CP^3 \ar[d]^{\pi} \\
	& \H\P^1=\S^4
	} 
\]
The map $\phi:\C^4_*\to \CP^3$ is the standard quotient projection. 
The map $\pi\circ \phi:\H^2_* \to \H\P^1$ associates to each quaternionic
line  $H \subset \H^2$ the corresponding point in the quaternionic $1$-dimensional projective space
$\H\P^1$, which is the $4$-sphere. Each complex line $\Lambda\subset \C^4=\H^2$ spans the unique
quaternionic line $H=\Lambda\oplus \jgot \Lambda  \subset \H^2$, and the space of all complex lines 
within a given quaternionic line (which may be identified with $\C^2$) is clearly parameterised by $\CP^1$.
This observation defines a real analytic fibre bundle $\pi:\CP^3\to \S^4$ with fibre $\CP^1$, 
called the {\em twistor map} or the {\em twistor projection}. The fibres of $\pi$ are projective lines 
$\CP^1\subset \CP^3$. We endow $\CP^3$ with the Fubini-Study metric and $\S^4$ with the spherical metric.
 
As shown by Bryant \cite[Theorem A]{Bryant1982JDG}, the complex hyperplane distribution
$\xi\subset T\CP^3$, where for every $p\in\CP^3$ the hyperplane $\xi_p$ is the orthogonal
complement of the tangent space at $p$ to the fibre $\pi^{-1}(\pi(p))$ in the Fubini-Study metric, 
is a holomorphic contact bundle given in suitable homogeneous coordinates by \eqref{eq:alpha0}.
Furthermore, the differential $d\pi_p:\xi_p \to T_{\pi(p)}\S^4$ for $p\in\CP^3$ is an isometry 
in the Fubini-Study metric on $\CP^3$ and the spherical metric on $\S^4$.

Among all minimal surfaces in $\S^4$ (and, more generally, in any smooth Riemannian $4$-manifold), 
there is a natural and important subclass consisting of {\em superminimal surfaces}. 
This term was introduced in 1982 by R.\ Bryant \cite{Bryant1982JDG},
although such surfaces had been studied much earlier. In particular, Bryant mentions several
works by E.\ Calabi and S.\ S.\ Chern from the period 1967--70 in which the authors 
exploited the fact that every minimal immersion of the 2-sphere into a higher-dimensional sphere is superminimal.  
A minimal immersion from a Riemann surface $M$ into $\mathbb S^4$ is superminimal if and only if a certain 
holomorphic quartic form on $M$ vanishes identically \cite[p.\ 466]{Bryant1982JDG} 
(on the 2-sphere $M=\S^2$ it always does).

We now recall a geometric characterization of superminimal surfaces in any smooth Riemannian $4$-manifold $(N,g)$,
due to T.\ Friedrich \cite{Friedrich1984,Friedrich1997} who pointed out that this class 
of minimal surfaces was first described by K.\ Kommerell in his 1897 dissertation. 
See also the brief historical survey in \cite{Forstneric2020JGEA}.

Assume that $M\subset N$ is a smooth embedded surface with the induced 
conformal structure in the Riemannian manifold $(N,g)$. 
(Our considerations will be of local nature, so they also apply to immersed surfaces.) 
Then, $TN|_M=TM\oplus \nu$ where $\nu$ is the orthogonal normal bundle of $M$ in $N$. 
A unit normal vector $n\in \nu_x$ at a point $x\in M$ determines a
{\em second fundamental form} $S_x(n):T_xM\to T_xM$, a self-adjoint linear operator. 
The surface $M$ is said to be {\em superminimal} if for every point $x\in M$ and tangent vector 
$0\ne v\in T_xM$, the curve 
\begin{equation}\label{eq:Ixv}
	I_x(v)=\bigl\{S_x(n)v: n\in \nu_x,\ |n|_g=1\bigr\} \subset T_xM
\end{equation}
is a circle centred at $0\in T_x M$, possibly reducing to the origin.  

If the surface $M$ and the ambient $4$-manifold $(N,g)$ are both oriented (so $M$ with this
orientation and the induced conformal structure is a Riemann surface), 
then one defines superminimal surfaces $M\subset N$ of positive or negative {\em spin} as follows. 
We coorient the normal bundle $\nu$ so that the orientations on $TM$ and $\nu$ add up 
to the orientation on $TN|_M=TM\oplus\nu$.
For $x\in M$ denote by $C_x$ the positively oriented unit circle in the oriented normal space $\nu_x$.
The nontrivial circles $I_x(v)$ of \eqref{eq:Ixv} are also positively oriented with respect to the orientation on $T_xM$. 
A superminimal surface $M\subset N$ is of {\em positive spin} if for
every point $x\in M$ and vector $0\ne v\ni T_xM$, the map $C_x\ni n \to S(n)v \in I_x(v)\subset T_xM$ 
is orientation preserving, and is of {\em negative spin} if this map is orientation-reversing. 
(This condition is irrelevant at points $x\in M$ where the circle $I_x(v)$ of \eqref{eq:Ixv} reduces to $0\in T_xM$.)
We denote by  $S_\pm(M,N)$ the spaces of conformal superminimal immersions of positive
or negative spin, respectively. Clearly, the spin gets reversed if we reverse the orientation on $N$.
(However, changing the orientation on $M$ also changes the coorientation on the normal bundle $\nu$, 
and hence the spin does not change.) In particular, the postcomposition by the antipodal map 
$x\mapsto -x$ on $\S^4$ (which is an orientation-reversing isometry) interchanges the spaces $S_\pm(M,\S^4)$.
For this reason, it suffices to consider superminimal surfaces in $\S^4$ of positive spin.

The following result is called the {\em Bryant correspondence} 
(see \cite[Theorems B, B', D]{Bryant1982JDG}). A generalization to more general
Riemannian $4$-manifolds is due to T.\ Friedrich \cite[Proposition 4]{Friedrich1984}; 
see also the summary statement \cite[Theorem 4.6]{Forstneric2020JGEA}.

%
%
\begin{theorem}[Bryant \cite{Bryant1982JDG}]\label{th:Bryant}
Let $\pi:\CP^3\to \S^4$ be the Penrose twistor bundle with the horizontal holomorphic contact 
subbundle $\xi\subset T\CP^3$ (orthogonal to the fibres of $\pi$). 
If $M$ is a Riemann surface and $X:M\to\CP^3$ is a holomorphic 
Legendrian immersion, then $\pi\circ X:M\to \S^4$ is a conformal superminimal immersion of positive spin.
Conversely, every conformal superminimal immersion $M\to \S^4$ of positive spin lifts to a unique
holomorphic Legendrian immersion $M\to \CP^3$.
\end{theorem}

Explicit formulas for the lifting can be found in \cite[Sect.\ 1]{Bryant1982JDG}, and a simpler
geometric description is given by T.\ Friedrich \cite{Friedrich1984}; see also \cite[Equation (4.3)]{Forstneric2020JGEA}. 
Uniqueness of a holomorphic Legendrian lifting of a superminimal surface is 
intimately related to the nonintegrability of the contact structure $\xi$.
In fact, every superminimal surface of positive spin in $\S^4$ admits precisely two Legendrian liftings in $\CP^3$, a
holomorphic and an antiholomorphic one, and these two liftings get interchanged by the antiholomorphic
involution $\iota:\CP^3\to \CP^3$,
\[
	\iota([z_1:z_2:z_3:z_4]) =  [-\bar z_2:\bar z_1:-\bar z_4:\bar z_3], 
\]
which preserves the fibres of $\pi:\CP^3\to \S^2$.  (See \cite[Sect.\ 6]{Forstneric2020JGEA} for more details.) 

By Theorem \ref{th:Bryant}, postcomposition by the twistor projection $\pi:\CP^3\to \S^4$ defines a homeomorphism 
\[
	\pi_*: \Lscr(M,\CP^3) \lra S_+(M,\S^4)
\] 
from the space of holomorphic Legendrian immersions $M\to\CP^3$ onto the space of 
superminimal immersions $M\to\S^4$ of positive spin, both endowed with the compact-open topology. 

We now describe some applications of the results on holomorphic Legendrian immersions into $\CP^3$, 
obtained in the previous sections, to superminimal surfaces in $\S^4$. 

%
%
\begin{corollary}[Runge approximation theorem for superminimal surfaces in $\S^4$]   \label{cor:RungeS4}
Let $M$ be a Riemann surface, either open or compact, and let $K$ be a compact subset of $M$.  
Every conformal superminimal immersion of positive spin from a neighbourhood of $K$ to $\S^4$ 
can be approximated uniformly on $K$ by complete superminimal immersions $Y:M\to\S^4$ of positive
spin. Furthermore, we may choose $Y$ to agree with $X$ to a given finite order at each point of a 
given finite subset of $K$. 
In particular, every Riemann surface immerses into the $4$-sphere as a complete conformal superminimal surface
of positive spin.
\end{corollary}

Since the antipodal map $x\mapsto -x$ on $\S^4$ interchanges the spaces $S_\pm(M,\S^4)$, the
corresponding result also holds for superminimal surfaces of negative spin in $\S^4$.

\begin{proof}
Let $X: U\to\S^4$ be a superminimal immersion of positive spin from a neighbourhood $U\subset M$ of $K$.
Fix a number $\epsilon>0$, a finite set $E\subset K$, and an integer $k\in\N$. 
By Theorem \ref{th:Bryant}, $X$ lifts to a holomorphic Legendrian immersion 
$F: U\to \CP^3$, i.e., $X=\pi\circ F$. By Theorem \ref{th:Runge-for-immersions} and
Corollary \ref{cor:completeLeg} we can approximate $F$ uniformly on $K$ by 
complete holomorphic Legendrian immersions $G: M\to\CP^3$ agreeing with $F$ to order $k$ at each point of 
$E$. (If $M$ is compact then every immersion from it is complete; the main point here concerns
open Riemann surfaces.) The projection $Y:=\pi\circ G: M\to\S^4$ is then a superminimal immersion 
(see Theorem \ref{th:Bryant}) that approximates $X$ on $K$ and agrees with $X$ to order $k$ 
at each point of $E$. Since $G$ is complete and the twistor projection is an isometry from the 
contact subbundle $\xi\subset T\CP^3$ onto $T\S^4$, $Y$ is complete as well. 
\end{proof}

Similarly one proves the following interpolation theorem. 

%
%
\begin{corollary}[Weierstrass interpolation theorem for superminimal surfaces in $\S^4$]  
\label{cor:WeierstrassS4}
Let $M$ be a Riemann surface, open or compact, and let $E$ be a closed discrete subset of $M$.  
Every map $E\to\S^4$ extends to a  complete superminimal immersion $M\to\S^4$ of positive spin.
\end{corollary}

From the Calabi--Yau theorem for Legendrian immersions to $\CP^3$
(see Corollary \ref{cor:CYLegendrian}) and the fact that $d\pi$ is an isometry on the 
contact subbundle $\xi\subset T\CP^3$ we infer the following.

%
%
\begin{theorem}[Calabi--Yau theorem for conformal superminimal surfaces in $\S^4$]\label{th:CYS4}
If $M$ is a compact bordered Riemann surface and $X:M\to\S^4$ is a superminimal immersion of
positive spin (defined on a neighbourhood of $M$ in an ambient Riemann surface), then $X$ can be 
approximated as closely as desired uniformly on $M$ by a continuous map 
$Y:M\to\S^4$ whose restriction to the interior $M^\circ=M\setminus bM$ is a complete, generically 
injective superminimal immersion of positive spin, and whose restriction to the boundary $bM$ is a 
topological embedding. In particular, $Y(bM)\subset\S^4$ is a union of pairwise disjoint Jordan curves. 
The analogous result holds for bordered surfaces with countably many boundary curves and 
without point ends (see the last part of Theorem \ref{th:abstractCY} for the precise statement).
\end{theorem}

\begin{proof}
This is seen by  the same argument as in the proof of Corollary \ref{cor:RungeS4};
however, we must justify the statement that $Y$ can be chosen generically injective
on $M$ and injective on $bM$. To this end, it suffices to show that at every step of the inductive 
construction, the Legendrian immersion $X_j:M\to\CP^3$ can be chosen such that 
the superminimal immersion $Y_j:=\pi\circ X_j:M\to\S^4$ is generically injective
on $M$ and injective on $bM$. If we approximate sufficiently closely at every step, 
then the limit map $Y=\lim_{j\to\infty}Y_j:M\to\S^4$ will enjoy the same properties.
(For the details in a similar setting, see \cite[proof of Theorem 1.1]{AlarconDrinovecForstnericLopez2015PLMS}.)

Let $X_0:M\to\CP^3$ be a holomorphic Legendrian immersion. We claim that there is an 
arbitrarily $\Cscr^1$ small holomorphic Legendrian perturbation $X$ of $X_0$ such that
$\pi\circ X:M\to \S^4$ is generically injective and $\pi\circ X:bM\to \S^4$ is injective; 
this will complete the proof. 

Pick a point $p\in \S^4$ which does not lie on the surface $\pi\circ X_0(M)\subset\S^4$ 
and choose Euclidean coordinates on $\S^4\setminus \{p\}=\R^4$. 
We associate to any map $X:M\to\CP^3$ uniformly close to $X_0$ the difference map
 $\delta X:M\times M\to \R^4$ defined by
\begin{equation}\label{eq:delta}
	\delta X(x,x')=\pi\circ X(x)-\pi\circ X(x') \in \R^4 \quad\text{for}\  x,x'\in M.
\end{equation}
Since $X_0$ is a Legendrian immersion, the map $\pi\circ X_0:M\to\S^4$ is an immersion
by Theorem \ref{th:Bryant}, and hence there is an open neighbourhood $U\subset M\times M$ of the diagonal 
$\Delta:=\{(x,x):x\in M\}$ such that $\delta X_0$ does not assume the value $0\in\R^4$ 
on $\overline U\setminus \Delta$. The same is then true for all maps sufficiently close to $X_0$ in 
$\Cscr^1(M,\CP^3)$. By the general position argument in 
\cite[proof of Lemma 4.4]{AlarconForstnericLopez2017CM}, 
a generic holomorphic Legendrian immersion $X:M\to\CP^3$ close to
$X_0$ in $\Cscr^1(M,\CP^3)$ is such that the difference map
$\delta X:M\times M\to \R^4$, and also its restriction $\delta X:bM\times bM\to \R^4$,
are transverse to the origin $0\in\R^4$ on $M\times M\setminus U$
and $bM\times bM\setminus U$, respectively. 
(The argument in \cite[Lemma 4.4]{AlarconForstnericLopez2017CM} is written for the standard contact
structure on $\CP^3$, but it applies in any complex contact manifold in view of the
Darboux neighbourhood theorem \cite[Theorem 1.1]{AlarconForstneric2019IMRN}.
Compare with \cite[proof of Theorem 1.2]{AlarconForstneric2019IMRN}.) 
Assume that $X$ is such. Since $\dim bM\times bM=2<4$, it follows that $\delta X$ does not assume 
the value $0\in\R^4$ on $bM\times bM\setminus \Delta$, which means that $\pi\circ X$ is injective on $bM$. 
Also, since $\dim M\times M=4$, transversality of $\delta X$ to $0\in\R^4$ on $M\times M\setminus U$
implies that $(\delta X)^{-1}(0)\subset M\times M$ consists of the diagonal 
$\Delta$ together with at most finitely many points in $M\times M\setminus \Delta$. 
\end{proof}

The following is an immediate consequence of Corollary \ref{cor:completeLeg} (iii)
and Theorem \ref{th:Bryant}.

%
%
\begin{corollary}
\label{cor:completesuperminimal}
Every open Riemann surface admits a complete superminimal immersion into $\S^4$ with dense image.
\end{corollary}

Finally, Theorem \ref{th:connected} on path connectedness of the space of Legendrian immersions
from any open Riemann surface into $\CP^3$ immediately implies the following.

%
%
\begin{corollary}\label{cor:connectedS4}
For every connected open Riemann surface, $M$, the spaces $S_\pm(M,\S^4)$ 
of superminimal immersions $M\to\S^4$ of positive resp.\ negative spin are path connected.
\end{corollary}


\subsection*{Acknowledgements}
A.\ Alarc\'on is supported by the State Research Agency (SRA) and European Regional Development Fund (ERDF) via the grant no.\ MTM2017-89677-P, MICINN, Proyecto PID2020-117868GB-I00 financiado por MCIN/AEI/10.13039/501100011033/ 
the Jun\-ta de Andaluc\'ia grant no. P18-FR-4049, and the Junta de Andaluc\'ia - FEDER grant no. A-FQM-139-UGR18, Spain. 
F.\ Forstneri\v c is supported  by the research program P1-0291 and the research grant 
J1-9104 from ARRS, Republic of Slovenia. F.\ L\'arusson is supported by Australian Research Council grant DP150103442.  
A part of the work on this paper was done while the second and the third named authors were visiting the 
University of Granada in September 2019. They wish to thank
the university for the invitation and partial support.



\vspace*{0.5cm}
\noindent Antonio Alarc\'{o}n\\
\noindent Departamento de Geometr\'{\i}a y Topolog\'{\i}a e Instituto de Matem\'aticas (IEMath-GR), Universidad de Granada, Campus de Fuentenueva s/n, E--18071 Granada, Spain\\
\noindent  e-mail: {\tt alarcon@ugr.es}

\vspace*{0.4cm}
\noindent Franc Forstneri\v c \\
\noindent Faculty of Mathematics and Physics, University of Ljubljana, Jadranska 19, SI--1000 Ljubljana, Slovenia\\
\noindent 
Institute of Mathematics, Physics and Mechanics, Jadranska 19, SI--1000 Ljubljana, Slovenia\\
\noindent e-mail: {\tt franc.forstneric@fmf.uni-lj.si}

\vspace*{0.4cm}
\noindent Finnur L\'arusson \\
\noindent School of Mathematical Sciences, University of Adelaide, Adelaide SA 5005, Australia \\
\noindent e-mail:  {\tt finnur.larusson@adelaide.edu.au}

\end{document}